\documentclass[11pt]{article}
\usepackage{etex}
\usepackage{amsmath,amssymb,amsfonts,amsthm}
\usepackage{a4wide}
\usepackage{color}
\usepackage{verbatim}
\usepackage{epsfig,subfigure,epstopdf,float}
\usepackage[]{graphics}
\usepackage{graphicx,inputenc}
\usepackage{subfigure}
\usepackage[justification=centering]{caption}
\usepackage{pictex}
\usepackage{wrapfig}
\usepackage{multirow}
\usepackage{booktabs}
\usepackage{algorithmic}

\usepackage[T1]{fontenc}
\usepackage{authblk}
\usepackage[textwidth=6.0in,textheight=9in]{geometry}
\usepackage[linesnumbered,ruled,vlined, ruled,vlined]{algorithm2e}

\usepackage[colorlinks,linktocpage,linkcolor=blue]{hyperref}
\usepackage{fancyhdr}

\usepackage[export]{adjustbox}

\newtheoremstyle{exampstyle}
  {\topsep} % Space above
  {\topsep} % Space below
  {\itshape} % Body font
  {} % Indent amount
  {\bfseries} % Theorem head font
  {.} % Punctuation after theorem head
  {.5em} % Space after theorem head
  {} % Theorem head spec (can be left empty, meaning `normal')

\theoremstyle{exampstyle}
\numberwithin{equation}{section}
\newtheorem{theorem}{Theorem}
\newtheorem{lemma}{Lemma}[section]
\newtheorem{assumption}[lemma]{Assumption}
\newtheorem{remark}[lemma]{Remark}

\newtheorem{corollary}[lemma]{Corollary}
\newtheorem{proposition}[lemma]{Proposition}

\allowdisplaybreaks
\usepackage{parskip}
\let\oldref\ref
\renewcommand{\ref}[1]{(\oldref{#1})}  % stupid kludge for laTeX
\renewcommand{\eqref}[1]{(\oldref{#1})}

\newbox\boxaddrone \newbox\boxaddrtwo

\newcommand{\norm}[1]{\left\Vert#1\right\Vert}

\newcommand{\A}{\mathcal{A}}

\newcommand{\N}{\mathcal{N}}

\newcommand{\ba}{\begin{eqnarray*}}
\newcommand{\ea}{\end{eqnarray*}}
\newcommand{\be}{\begin{equation}}
\newcommand{\ee}{\end{equation}}
\newcommand{\bea}{\begin{eqnarray}}
\newcommand{\eea}{\end{eqnarray}}

\newbox\boxaddrone \newbox\boxaddrtwo

\def\N+{n\in\mathbb{N}^{+}}

\def\n{\partial{\overrightarrow{\bf n}}}
\def\1d{\mathcal{D}((-\Delta)^{\gamma_1+1/2})}
\def\2d{\mathcal{D}((-\Delta)^{\gamma_2+1})}

\def\l{\langle}
\def\ro{\rangle_{\Omega}}
\def\rp{\rangle_{\kappa,\partial\Omega}}

\def\s{\text{Span}}

\def\A{\mathcal{A}}

\begin{document}
\title{Conditional well-posedness and data-driven method for identifying the dynamic source in a coupled diffusion system from one single boundary measurement}

%\author[1,3]{Jijun Liu\thanks{jjliu@seu.edu.cn}}
\author[1,2]{Chunlong Sun\thanks{sunchunlong@nuaa.edu.cn}}
\author[3]{Mengmeng Zhang\thanks{corresponding author: mmzhang@hebut.edu.cn}}
\author[4]{Zhidong Zhang\thanks{zhangzhidong@mail.sysu.edu.cn}}
\affil[1]{School of Mathematics, Nanjing University of Aeronautics and Astronautics, Nanjing 211106, Jiangsu, China}
\affil[2]{Nanjing Center for Applied Mathematics, Nanjing 211135, Jiangsu, China}
\affil[3]{School of Science, Hebei University of Technology, Tianjin 300401, China}
\affil[4]{School of Mathematics (Zhuhai), Sun Yat-sen University, Zhuhai 519082, Guangdong, China}

\maketitle

\begin{abstract}
This work considers the inverse dynamic source problem arising from the time-domain fluorescence diffuse optical tomography (FDOT). We recover the dynamic distributions of fluorophores in biological tissue by the one single boundary measurement in finite time domain. We build the uniqueness theorem of this inverse problem. After that, we introduce a weighted norm and establish the conditional stability of Lipschitz type for the inverse problem by this weighted norm.
The numerical inversions are considered under the framework of the deep neural networks (DNNs).
We establish the generalization error estimates rigorously derived from Lipschitz conditional stability of inverse problem. Finally, we propose the reconstruction algorithms and give several numerical examples illustrating the performance of the proposed inversion schemes.\\


\end{abstract}

\section{Introduction.}
\setcounter{equation}{0}
\subsection{Background and mathematical model.}
Fluorescence diffuse optical tomography (FDOT) is one type of diffuse optical tomography that uses fluorescence light from fluorophores in biological tissue, which is rapidly gaining acceptance as an important diagnostic and monitoring tool of symptoms in medical applications \cite{Arridge99, Arr09}. The fluorescence contrast agents allow tracking non-invasively and quantitatively specific molecular events or provide some clinically important information \textit{in vivo}. For FDOT, two processes are coupled, namely, {\it excitation} and {\it emission} ({\it fluorescence}), which can be described by a coupled diffusion system as follows.
%the Boltzmann radiative transfer equation or shortly the radiative transfer equation (RTE). However, the cost of solving RTE is extremely high, and it is usually replaced by the diffusion equation (DE). The basic idea of the diffusion approximation is: when scattering is much stronger than absorption, the radiant intensity can be expressed as an isotropic photon density plus a small photon flux, and sequentially the transport equation can be reduced to a DE.

%The excitation photons injected (at initial time $t=0$) on the boundary of the medium propagate to the fluorophores and then some photons are absorbed by them which excite the fluorophore molecules. After a moment of the absorption (at $t=0>0$), the fluorophores emit other photons, fluorescence, at more longer wavelength than the wavelength of the excitation photons.

%However, as an approximation of the RTE, the DE can not %approximate the RTE well, when the time $t$ is sufficiently small or large. In fact, the fluorophores need after a moment (at $t=0>0$) to emit the emission (fluorescence) photons, after the injection of excitation photons (at $t=0$). Hence we formulate our DE-based FDOT problem as follows.
Let $\Omega\subset\mathbb R^d$ be the background medium with its boundary $\partial\Omega$. Let $u_e, u_m$ be the photon density of excitation light and emission light, respectively. Furthermore, we consider the time-dependent fluorophores which means that the absorption coefficient $\mu_f$ depends on both $x$ and $t$. Then we consider the following initial boundary value problems for $u_e,\,u_m$ with $T< \infty$:
\begin{equation}\label{excitation-Ue}
\begin{cases}
\begin{aligned}
\left(c^{-1}\partial_t +\A\right)u_e(x,t)&=0, && (x,t)\in\Omega\times (0,T),\\
u_e&=0, && (x,t)\in\Omega\times \{0\}, \\
\mathcal B u_e&=g(x,t), && (x,t)\in\partial\Omega\times(0,T),
\end{aligned}
\end{cases}
\end{equation}
and
\begin{equation}\label{emission-um}
\begin{cases}
\begin{aligned}
\left(c^{-1}\partial_t +\A \right) u_m(x,t)&=\mu_f(x,t)u_e(x, t)=:S[\mu_f,u_e], && (x,t)\in\Omega\times (0,T),\\
u_m&=0, && (x,t)\in\Omega\times \{0\},\\
\mathcal B u_m&=0, && (x,t)\in\partial\Omega\times(0,T),
\end{aligned}
\end{cases}
\end{equation}
where $\mu_f$ denotes the dynamic distributions of fluorophores inside $\Omega$. The elliptic operator $\mathcal{A}$ and  boundary condition $\mathcal B$ are defined as
\begin{equation}\label{A}
\begin{cases}
\begin{aligned}
   \A \psi &=-\nabla\cdot(\kappa(x)\nabla \psi)+\mu_a(x)\psi,\ \psi\in H^2(\Omega),\\
    \mathcal B \psi&=\Big( \frac{\partial \psi}{\n}+ \beta \psi\Big) \Big|_{\partial\Omega},
\end{aligned}
\end{cases}
\end{equation}
where $\beta>0$ is the Robin coefficient;  $\kappa(x)$ is the diffusion coefficient; $\mu_a$ is the absorption coefficient of background medium. The time-domain FDOT is a method to achieve imaging of $\mu_f$ from the boundary data given by
\begin{equation}\label{data}
\varphi(x,t)=\frac{\partial u_m}{\n}\Big|_{\Gamma \times(0,T)},
\end{equation}
where  $\Gamma$ is a nonempty open subset of boundary $\partial\Omega$.

In the aspect of theoretical analysis, the uniqueness of recovering $\mu_f(x,t)$ requires the whole information of $u_m$ on $\Omega\times(0,T)$, but it is impractical in applications. However, although the boundary measurement as in \eqref{data} is practical, it is not sufficient to uniquely reconstruct $\mu_f(x,t)$ and its uniqueness is almost open. Actually, even for recovering the stationary fluorophores (i.e., $\mu_f(x,t)\equiv\mu_f(x)$) from \eqref{data}, the literature on the uniqueness is relatively rare. An identifiability result of absorption coefficient $\mu_f(x)$ is established in \cite{Liu:2020}, but it requires strong prior assumptions on $\mu_f(x)$. The authors prove that for $\Omega=\mathbb{R}_+^3$, $x=(\tilde x, x_3)\in \mathbb{R}^2\times\mathbb{R}_+^1$, by supposing $\mu_f(x)$ has the variable separable form $\mu_f(x)=p(\tilde x)q(x_3)$ with known vertical information $q(x_3)$, the horizontal information $p(\tilde x)$ can be uniquely determined from the boundary measurement \eqref{data} with $\Gamma=\partial\Omega$. Assuming that the support of $\mu_f(x)$ is at a point, called a point target, the minimal number of observation detectors to determine the point target location as well as its local stability are investigated in \cite{Eom24}.

In this work, we study the inverse problem of recovering the dynamic source $\mu_f(x,t)$ from boundary data \eqref{data}. Suppose that $S[\mu_f,u_e](x,t;x_s)$ in  \eqref{emission-um} possesses the following semi-discrete formulation:
\begin{equation}\label{ua-time}
S[\mu_f,u_e]=\sum_{k=1}^K p_k(x)\chi_{{}_{t\in [t_{k-1},t_k)}},
\end{equation}
where $\chi$ is the indicator function, the time mesh $\{t_k\}_{k=0}^K$ is given and the spatial components $\{p_k(x)\}_{k=1}^K$ are undetermined. For the unknown $\{p_k(x)\}_{k=1}^K$, we consider to recover them in $L^2(\Omega)$. Since the excitation $u_e$ is known, after solving $\{p_k(x)\}_{k=1}^K$, the values of $\{\mu_f(\cdot,t_k)\}_{k=0}^{K-1}$ can be recovered as $\mu_f(\cdot, t_k)={p_{k+1}(\cdot)}/{u_e(\cdot,t_k)}, \ k=0\cdots,K-1$. 
% Since the excitation $u_e$ is known, we can extract the information of $\mu_f$ from the components $\{p_k(x)\}_{k=1}^K$. 
We conclude the inverse problem of time-domain FDOT as follows:
\begin{equation}\label{IP}
\text{recovering $\{p_k(x)\}_{k=1}^K$  in \eqref{ua-time} from the one single boundary measurement \eqref{data}}.
\end{equation}

\subsection{Literature and outline.}
The inverse problem \eqref{IP} has not been rigorously investigated as far as we know. With the Laplace transform and the knowledge of complex analysis, \cite{LinZhangZhang:2022, Sun2022} established the uniqueness theorem of inverse problem \eqref{IP} but required $T=\infty$. In this work, we will consider the uniqueness of recovering $\{p_k(x)\}_{k=1}^K$ by the boundary measurement in finite time domain ($T<\infty$). Furthermore, we will define a weighted norm and establish the conditional stability of Lipschitz type for the inverse problem \eqref{IP} by the defined weighted norm. The defined weighted norm is a little bit weaker than the standard $L^2$ norm. Hence the obtained stability result will be the conditional stability for our inverse problem in $L^2(\Omega)$ by some weighted norm.

Numerically, for the forward problem of FDOT, F.Marttelli et al \cite{Marttelli} and H.B.Jiang \cite{Jiang11} summarized the principles and applications of light propagation through biological tissue and other diffusive media, and gave the theory, solutions and software codes, respectively. For the inverse problem of FDOT, the necessary regularization techniques such as Tikhonov regularization, sparse regularization methods and hybrid regularization methods have been introduced to overcome the ill-posedness of the inverse problems \cite{Correia10,Dutta12,Liu20}. To solve the inverse problem, the gradient-type iteration such as Landweber iteration \cite{Hanke_1995} and the Newton-type iteration such as Levenberg–Maquardt method  \cite{Johann2010,hanke1997,Levenberg1944,Marquardt1963}, the iteratively regularized Gauss-Newton method \cite{JinQ_2013,Kaltenbacher_2015} are common iterative methods. \cite{Arridge99, Arr09} presented a review of methods for the forward and inverse problems in optical tomography. However, most existing works focus on recovering the stationary $\mu_f(x)$, the inversion of dynamic $\mu_f(x,t)$ in \eqref{IP} is still a challenging problem.

%In \cite{Lam05}, the time-resolved FDOT was considered but the fundamental solution was used simply by ignoring the presence of boundaries. Many of these works are focusing on small animal measurements and employing the trans-illumination scheme.  We are here focusing on a epi-illumination scheme of the detection for more larger tissues. We summarize our main results below.

Recently, deep learning methods for solving PDEs have been realized as an effective approach, especially in high dimensional PDEs. Such methods have the advantage of breaking the curse of dimensionality. The basic idea is to use neural networks (nonlinear functions) to approximate the unknown solutions of PDEs by learning the networks parameters. For the forward problems, there exists many numerical works with deep neural networks involving the deep Ritz method (DRM) \cite{EWN-DRM}, the deep Galerkin method (DGM) \cite{Sirignano-DGM}, the DeepXDE method \cite{Lu-DeepXDE}, deep operator network method (DeepONet) \cite{Lu-DeepONet-2}, physical information neural networks (PINNs) \cite{Raissi}, the weak adversary neural network (WAN) \cite{Bao2, Bao1} and so on.
Theoretically, there are some rigorous analysis works investigating the convergence and error estimates for the solution of PDEs via neural networks, but the result are still far from complete. For example, the convergence rate of DRM with two layer networks and deep networks were studied in \cite{Duan_2022,Hong_2021, Lu_2021,Luo_2020}; the convergence of PINNs was given in \cite{Jiao,Mishra2,de2023error,Shin1,Shin2}. For the inverse problems, the PINNs frameworks can be employed to solve the so-called data assimilation or unique continuation problems, and rigorous estimates on the generalization error of PINNs were established in \cite{Mishra2}.
Bao et al \cite{Bao1} developed the WAN to solve electrical impedance tomography (EIT) problem. In \cite{Zhang-Li-Liu}, the authors studied a classical linear inverse source problem using the final time data under the frameworks of neural networks, where a rigorous generalization error estimate is proposed with a novel loss function including the Sobolev norm of some residuals. For more specific inverse problems applied in engineering and science, we refer to \cite{Chen-Lu,Raissi-Yazdani,Shukla}.

In this work, we consider the reconstruction of the dynamic source of the inverse problem \eqref{IP} parameterized by deep neural networks (DNNs). For the emission process, a new loss function is proposed with regularization terms depending on the derivatives of the residuals for PDEs and measurement data. We establish generalization error estimates rigorously derived from Lipschitz conditional stability of inverse problems. Furthermore, we propose the reconstruction scheme and make some numerical implementations to demonstrate the validity of the reconstruction algorithm.

The rest of this article is organized as follows. In Section \ref{sec_pre}, we collect several preliminary works. We prove the uniqueness theorem and the condition stability in Section \ref{section_unique}, which are stated in Theorems \ref{theorem-uniq} and \ref{stability-inv}. The numerical inversions will be considered in Section \ref{sec_pinn}. The authors introduce the loss functions and build generalization error estimates in Theorem \ref{theorem-error}. Then the reconstruction algorithms and some numerical experiments are provided in Section \ref{sec_num}, which can verify the theories.

\section{Preliminaries.}\label{sec_pre}
\subsection{Eigensystem of operator $\A$.}
For the operator $\A$ on $H^2(\Omega)$ with Robin boundary condition, we denote the eigensystem by $\{\lambda_n,\varphi_n(x)\}_{n=1}^\infty$. Then the following properties will be valid:
\begin{itemize}
    \item $0<\lambda_1\le \lambda_2\le\cdots$ and $\lambda_n\to\infty$ as $n\to \infty$;
    \item $\{\varphi_n\}_{n=1}^\infty\subset H^2(\Omega)$ is an orthonormal basis of $L^2(\Omega)$.
\end{itemize}
Furthermore, if $\varphi_n$ is an eigenfunction of $\A$ corresponding to $\lambda_n$, so is $\overline{\varphi_n}$,  where $\overline{\varphi_n}$ is the complex conjugate of $\varphi_n$. Hence we have that the set $\{\varphi_n(x)\}_{n=1}^\infty$ coincides with $\{\overline{\varphi_n(x)}\}_{n=1}^\infty$.
The trace theorem yields that   $\{\frac{\partial\varphi_n}{\n}|_{\partial \Omega}\}_{n=1}^\infty\subset H^{1/2}(\partial \Omega)$.
Also, we denote $\l\cdot,\cdot\rp$ as the weighted inner product in $L^2(\partial\Omega)$.

The next lemmas concern the vanishing property and the density of $\frac{\partial\varphi_n}{\n}$ on $\partial\Omega$.
\begin{lemma}\label{nonempty_open}
If $\Gamma$ is a nonempty open subset of $\partial \Omega$, then for each $n\in\mathbb N^+$, $\frac{\partial \varphi_n}{\n}$ can not vanish on $\Gamma$.
\end{lemma}
\begin{proof}
 See \cite[Lemma 2.1]{LinZhangZhang:2022}.
\end{proof}

\begin{lemma}
 The set $\text{Span}\{\frac{\partial\varphi_n}{\n}|_{\partial \Omega}\}_{n=1}^\infty$ is dense in $L^2(\partial \Omega)$.
\end{lemma}
\begin{proof}
 Not hard to see that $H^{3/2}(\partial\Omega)$ is dense in  $L^2(\partial \Omega)$ under the norm $\|\cdot\|_{L^2(\partial\Omega)}$. So it is sufficient to show $\tilde\psi\in H^{3/2}(\partial \Omega)$ vanishes almost everywhere on $\partial \Omega$ if $\l \tilde\psi,\frac{\partial\varphi_n}{\n}\rp=0$ for $n\in\mathbb{N}^+$.

We set $\psi$ be the weak solution of the system below:
\begin{equation*}
\begin{cases}
\begin{aligned}
   \A \psi(x)&=0, &&x\in \Omega,\\
    \mathcal B\psi&=\tilde \psi,&&x\in\partial \Omega.
\end{aligned}
\end{cases}
\end{equation*}
We have $\psi\in H^2(\Omega)$ from the regularity $\tilde\psi\in H^{3/2}(\partial \Omega)$, sequentially the Green's identity can be used. For $n\in\mathbb N^+$, we have
\begin{align*}
\l \A\psi,\varphi_n\ro-\l \psi, \A\varphi_n\ro
&=\l \psi,\frac{\partial\varphi_n}{\n}\rp-\l\frac{\partial\psi}{\n}, \varphi_n\rp\\
&=\l \psi+\beta\frac{\partial\psi}{\n},\frac{\partial\varphi_n}{\n}\rp\\
&=\l \tilde\psi,\frac{\partial\varphi_n}{\n}\rp.
\end{align*}
From $\A \psi=0$ on $\Omega$ and the fact $\l \tilde\psi,\frac{\partial\varphi_n}{\n}\rp=0$, we have $$\l\psi,\A\varphi_n\ro =\lambda_n\l \psi,\varphi_n\ro=0.$$
So we have proved that for each $n\in\mathbb N^+$, $\l\psi, \varphi_n\ro=0$. Recalling the completeness of $\{\varphi_n\}_{n=1}^\infty$ in $L^2(\Omega)$, it holds that  $\|\psi\|_{L^2(\Omega)}=0$.
From the definition of weak derivative and Sobolev space, we have $\|\psi\|_{H^2(\Omega)}=0$. By the  continuity of the trace operator, it gives that
$$\|\psi\|_{L^2(\partial \Omega)}\le C\|\psi\|_{H^2(\Omega)}=0,\quad
\Big\|\frac{\partial\psi}{\n}\Big\|_{L^2(\partial\Omega)}
\le C\|\psi\|_{H^2(\Omega)}=0.$$
This means that $\tilde\psi=0$ almost everywhere on $\partial\Omega$ and the proof is complete.
\end{proof}

\subsection{The auxiliary functions $\{\xi_l\}_{l=1}^\infty$ and the coefficients $\{c_{z,n}\}$.}
From the above lemma, we are allowed to construct the orthonormal basis $\{\tilde \xi_l\}_{l=1}^\infty$ in $L^2(\kappa,\partial \Omega)$. Firstly we set $\tilde\xi_1=\frac{\partial\varphi_1}{\n}|_{\partial \Omega}/\|\frac{\partial\varphi_1}{\n}\|_{L^2(\kappa,\partial \Omega)},$ and assume that the orthonormal set $\{\tilde\xi_j\}_{j=1}^{l-1}$ has been built for $l=2,3,\cdots$. Then we set $n_l\in\mathbb N^+$ be the smallest number such that  $\frac{\partial\varphi_{n_l}}{\n}|_{\partial \Omega}\notin \s\{\tilde\xi_j\}_{j=1}^{l-1}$, and pick $\tilde\xi_l\in \s\{\frac{\partial\varphi_{n_l}}{\n}|_{\partial \Omega},\ \tilde\xi_1, \cdots, \tilde\xi_{l-1}\}$ satisfying
\begin{equation*}
 \l\tilde\xi_l, \tilde\xi_j\rp=0\ \text{for}\ j=1,\cdots, l-1,\ \text{and}\ \|\tilde\xi_l\|_{L^2(\kappa,\partial \Omega)}=1.
\end{equation*}
The density of $\text{Span}\{\frac{\partial\varphi_n}{\n}|_{\partial \Omega}\}_{n=1}^\infty$ in $L^2(\kappa,\partial \Omega)$ yields that $\{\tilde\xi_l\}_{l=1}^\infty$ is an orthonormal basis in $L^2(\kappa,\partial \Omega)$. Also, we have
$\tilde \xi_l\in H^{1/2}(\partial \Omega)$ for each $l\in\mathbb N^+$.

Next, for $l\in \mathbb N^+$, we define $\xi_l\in H^1( \Omega)$ be the weak solution of the system:
\begin{equation}\label{PDE_xi}
\begin{cases}
\begin{aligned}
  \A\xi_l(x)&=0, &&x\in \Omega,\\
    \mathcal B\xi_l&=\tilde \xi_l,&&x\in\partial \Omega.
\end{aligned}
\end{cases}
\end{equation}
Fixing $z\in\partial \Omega$, we define the series
$\psi_z^N\in H^1(\Omega)$ as
\begin{equation}\label{psi_z^N}
 \psi_z^N(x)=\sum_{l=1}^N \tilde\xi_l(z)\overline{\xi_l(x)} , \ x\in  \Omega.
\end{equation}
The finite summation $\psi_z^N$ is constructed following the role of Dirac delta function, which reflects the information of the targeted function on a specific point via integration.
With the following lemma, we can give the coefficients  $\{c_{z,n}\}$.
\begin{lemma}\label{c_z,n}
 For each $z\in\partial \Omega$ and $n\in \mathbb N^+$, $\lim_{N\to \infty}\l \psi_z^N,\overline{\varphi_n}\ro$ exists and we denote the limit by $c_{z,n}$.
\end{lemma}
\begin{proof}
 From Green's identities we have
 \begin{equation*}
 \begin{aligned}
  \l \psi_z^N,\overline{\varphi_n}\ro&=\lambda_n^{-1}\l\psi_z^N, \A\overline{\varphi_n}\ro \\
  &=\lambda_n^{-1}\Big( \l \A\psi_z^N,\overline{\varphi_n}\ro-\l\psi_z^N,\frac{\partial\overline{\varphi_n}}{\n}\rp+\l\frac{\partial\psi_z^N}{\n},\overline{\varphi_n}\rp\Big) \\
  &=-\lambda_n^{-1}\l\beta\frac{\partial\psi_z^N}{\n}+\psi_z^N,\frac{\partial\overline{\varphi_n}}{\n}\rp\\
  &=-\lambda_n^{-1}\sum_{l=1}^N \tilde\xi_l(z)\l\frac{\partial\varphi_n}{\n}, \tilde\xi_l\rp=:c_{z,n}^N,
  \end{aligned}
 \end{equation*}
where the system \eqref{PDE_xi} and the boundary condition of $\varphi_n$ are used. From the definition of $\{\tilde\xi_l\}_{l=1}^\infty$, we have   $\l\frac{\partial\varphi_n}{\n}, \tilde\xi_l\rp=0$ for large $l$. Hence the value of $c_{z,n}^N$ will not change if $N$ is sufficiently large. This gives that $\lim_{N\to \infty}\l \psi_z^N,\overline{\varphi_n}\ro$ exists and the proof is complete.
\end{proof}

Then we consider the following system and denote the solutions by $\{u_z^N\}$:
\begin{equation}\label{u_z^N}
\begin{cases}
 \begin{aligned}
  (c^{-1}\partial_t+\A)u_z^N(x,t)&=0, &&(x,t)\in \Omega\times(0,\infty),\\
 \mathcal Bu_z^N(x,t)&=0, &&(x,t)\in \partial \Omega\times(0,\infty),\\
  u_z^N (x,0)&=-\psi_z^N, &&x\in \Omega.
 \end{aligned}
\end{cases}
\end{equation}

With the coefficients $\{c_{z,n}\}$ and defining $p_{k,n}:=\l p_k(\cdot),\varphi_n(\cdot)\ro$, we give the next assumptions, which will be used in the proof of uniqueness and conditional stability.
\begin{assumption}\label{condition_f_regularity}
\hfill
\begin{itemize}
\item [(a)] For $k\in\{1,\cdots,K\}$ and a.e. $z\in\partial\Omega$, there exists $C>0$ which is independent of $N$ such that $\sum_{n=1}^\infty  |c^N_{z,n}p_{k,n}|<C<\infty$ for each $N\in\mathbb N^+$.
\item [(b)] For a.e. $z\in \partial \Omega$, it holds that $\sum_{k=1}^K\sum_{n=1}^\infty |c_{z,n}p_{k,n}|<\infty$.
\item [(c)] For a.e. $t\in[0,\infty)$, the spatial components $\{p_k(x)\}_{k=1}^K$ are smooth enough such that the series
$\sum_{l=1}^\infty\tilde\xi_l(x)\l\frac{\partial u}{\n}(\cdot, t),\tilde\xi_l(\cdot)\rp$ converges to $\frac{\partial u}{\n}(x, t)$ pointwisely for a.e. $x\in\partial\Omega$.
\end{itemize}
\end{assumption}

\subsection{General settings of deep neural networks.}

For better readability, we introduce some general settings of deep neural networks. Numerically, we parameterize the excitation $u_e$, emission $u_m$ as well as the unknown $\mu_f$ by deep neural networks, given by $u_{e,\theta_e}$, $u_{m,\theta_m}$ and $\mu_{f,\theta_f}$, respectively. Here $\theta_e$, $\theta_m$ and $\theta_f$ are the corresponding network parameters. The loss functions to determine the optimal network parameters and related generalization error estimates will be investigated in Section \ref{sec_pinn}, whereas we give the network architecture and the quadrature rule in computing the loss functions.

%\blue{We solve the inverse problem \eqref{IP} using deep learning method. Firstly, we employ PINNs to seek the solution of initial boundary value problem \eqref{excitation-Ue}. We parameterize $u_e$ by deep neural networks $u_{e,\theta_e}$ with network parameters $\theta_e$. The optimal network parameters $\theta^*_e$ are selected by minimizing the loss function under the PINNs frameworks. The numerical solution to initial boundary value problem \eqref{excitation-Ue} is given by $u_e^*(x,t):=u_{e,\theta_e^{*}}(x,t)$. After that, we will construct data-driven methods to solve the inverse problem \eqref{IP}. Similarly, we parameterize $\mu_f$ and $u_m$ by deep neural networks $\mu_{f,\theta_f}$ and $u_{m,\theta_m}$
%with network parameters $\theta_f$ and $\theta_m$ respectively. Let $(\theta^*_f,\theta^*_m)$ are the optimal network parameters.  The approximations of $(\mu_f,u_m)$ for inverse problem \eqref{IP} will be given as $\mu^*_{f}:=\mu_{f,\theta^*_f}$, $u^*_{m}:=u_{m,\theta^*_m}$.}

%We design an appropriate loss function, which is minimized to determine the optimal parameters $(\theta^*_f,\theta^*_m)$.
%An gradient-based method is applied to alternately update the network parameters so that $(\mu^*_{f}, u^*_{m})$ with $\mu^*_{f}:=\mu_{f,\theta^*_f}$, $u^*_{m}:=u_{m,\theta^*_m}$ gradually approximates $(\mu_f,u_m)$ for the inverse problem.

{\bf{Network architecture.}} Noting that $u_{e,\theta_e}$, $u_{m,\theta_m}$ and $\mu_{f,\theta_f}$ are several separate networks with  arguments $(x,t)$.
Thus, we use $\xi$ to denote collectively the network parameters which are bounded for a parametric function $s_\xi(z)$
such that a general scheme can be applied for either $u_{e,\theta_e}(x,t)$, $u_{m,\theta_m}(x,t)$
or $\mu_{f,\theta_f}(x,t)$ (with $z=(x,t),\xi=\theta$).
% \QL{Check my edits of the above. DIdn't understand this the first time. See if this is correct understanding}

For a positive integer $K\in \mathbb{N}$, a $K$-layer feed-forward neural
network of $s_\xi(z)$ for $z\in\mathbb{R}^{d_0}$ is a function defined by
\begin{eqnarray}\label{DeepNN}
s_\xi(z):=W_{K} l_{K-1} \circ \cdots \circ l_{1}(z)+b_{K},
\end{eqnarray}
where the $k$-th layer $l_{k}: \mathbb{R}^{d_{k-1}} \rightarrow \mathbb{R}^{d_{k}}$
is given by $l_{k}(z)=\sigma\left(W_{k} z+b_{k}\right)$ with weights
$W_{k} \in \mathbb{R}^{d_{k} \times d_{k-1}}$ and biases
$b_{k} \in \mathbb{R}^{d_{k}}$ for $k=2, \cdots, K,$ and $\sigma(\cdot)$ is the activation function.
The network parameters of all layers are collectively denoted by
\begin{eqnarray*}
\xi:=\left(W_{K}, b_{K}, W_{K-1}, b_{K-1}, \cdots, W_{1}, b_{1}\right).
\end{eqnarray*}

{\bf{Quadrature rules \cite{Mishra2}.}} We also need to state the quadrature rules applied for approximately computing the integrals in the loss function. Consider a mapping $h: \mathbb{D}\subset \mathbb{R}^{\bar d} \mapsto \mathbb{Y}\subset \mathbb{R}^{m}$. To approximate the integral
 \begin{eqnarray*}
 \bar{h}:=\int_{\mathbb{D}} h(y) d y
 \end{eqnarray*}
 numerically by a quadrature rule with $d y$ being the $\bar{d}-$dimensional Lebesgue measure,
 we need the quadrature points $y_{i} \in \mathbb{D}$ for $i=1,2,\cdots, N$
 as well as the weights $w_{i}\in \mathbb{R}_{+}$. Then a quadrature formula is defined by
 \begin{eqnarray*}\label{integral}
 \bar{h}_{N}:=\sum_{i=1}^{N} w_{i} h\left(y_{i}\right).
 \end{eqnarray*}
 We further assume that the quadrature error is
 \begin{eqnarray}\label{quadrature}
 \left|\bar{h}-\bar{h}_{N}\right| \leqslant C_{q}\left(\|h\|_{\mathbb{Y}}, \bar{d}\right) N^{-\alpha}
 \end{eqnarray}
 for some $\alpha>0$, which is well-known for several classical quadrature formulas.
 %More precisely,
% for the domain $\mathbb{D}$  in a reasonably low dimensional space,
% i.e., $\bar{d} \le 4$, one can use standard (composite) Gauss quadrature rules on an underlying grid.
% In this case, the quadrature points and weights depend on the order of the quadrature rule \cite{Stoer}
% and the rate $\alpha$ depends on the regularity of the underlying integrand, i.e., on the space $\mathbb{Y}$.
% On the other hand, these grid-based quadrature rules are not suitable for domains in high dimensions.
% For moderately high dimensions, i.e., $4 \le \bar{d} \le 20$, we can use low discrepancy sequences,
% such as the Sobol and Halton sequences, as quadrature points \cite{Caflisch}.
% % As long as the integrand $h$ is of bounded Hardy-Krause variation \cite{Owen},
% % the error in (\ref{quadrature}) converges at a rate $(\log (N))^{\bar{d}} N^{-1}$.
% % One can also employ sparse grids and Smolyak quadrature rules \cite{Bungartz} in this regime.
% For problems in very high dimensions with $\bar{d} \gg 20$,
% Monte-Carlo quadrature is the numerical integration method of choice \cite{Caflisch}.
% In this case, the quadrature points are chosen randomly,
% and are usually independently and identically distributed with respect to a scaled Lebesgue measure.
% % The estimate (\ref{quadrature}) holds in the root mean square (RMS)
% % \QL{If not using RMS again, don't need to introduce short form}
% % sense and the rate of convergence is $\alpha=1/2$.
%

\section{The uniqueness and conditional stability.}\label{section_unique}
\subsection{Representation of the forward operator.}
To investigate the inverse problem, we firstly give the representation formula of the forward operator, which maps the unknown source to the boundary flux data. With a small $\epsilon>0$ we give the following equation
\begin{equation}\label{PDE2}
 \begin{cases}
  \begin{aligned}
  (c^{-1}\partial_t+\A)u(x,t)&=p(x),&&  (x,t)\in \Omega\times(0,\epsilon),\\
   u(x,t)&=0, &&(x,t)\in \partial\Omega\times(0,\epsilon)\cup \Omega\times\{0\},
  \end{aligned}
 \end{cases}
 \end{equation}
and prove the next lemma for equation \eqref{PDE2}.

\begin{lemma}
We define $w_z^N=u_z^N+\psi_z^N$, where $u_z^N$ and $\psi_z^N$ satisfy \eqref{u_z^N} and \eqref{psi_z^N} respectively. Then for the solution $u$ of equation \eqref{PDE2}, we have
 \begin{equation*}
  -\frac{\partial u}{\n}(z,t)=\lim_{N\to\infty}\int_\Omega p(x) w_z^N(x,t)\ dx.
 \end{equation*}
\end{lemma}
\begin{proof}
 From the definition of $\psi_z^N$, we see that
 $ (c^{-1}\partial_t+\A)\psi_z^N=0.$ This result and \eqref{u_z^N} show that $w_z^N$ satisfies the equation
\begin{equation*}
 (c^{-1}\partial_t+\A)w_z^N=0, \  (x,t)\in \Omega\times(0,T),
\end{equation*}
with the boundary condition $w_z^N|_{\partial \Omega}=\psi_z^N|_{\partial \Omega}$ and the initial condition $w_z^N(x,0)=0$. Therefore, Green's identities give that for each $v\in H^1(\Omega)$ with the boundary condition $\mathcal B v=0$,
\begin{equation*}
 \int_\Omega (c^{-1}\partial_t+q)w_z^N(x,t) v(x)
 +\kappa\nabla w_z^N(x,t) \cdot \nabla v(x)\ dx
 -\int_{\partial \Omega} \kappa v(x)\frac{\partial w_z^N}{\n}(x,t)\ dx=0,\ \ t\in (0,T).
\end{equation*}

From equation \eqref{PDE2}, we obtain
\begin{equation*}
 \begin{aligned}
  \int_0^t \int_\Omega p(x) w_z^N(x,t-\tau)\ dx\ d\tau=\int_0^t\int_\Omega (c^{-1}\partial_t+\A)u(x,\tau)\ w_z^N(x,t-\tau)\ dx\ d\tau.
 \end{aligned}
\end{equation*}
Green's identities and the vanishing initial conditions of $u$ and $w_z^N$ give that
\begin{equation*}
 \begin{aligned}
 \int_0^t\int_\Omega (c^{-1}\partial_t+q) u(x,\tau)\ w_z^N(x,t-\tau)\ dx\ d\tau
 =& \int_0^t\int_\Omega (c^{-1}\partial_t+q) w_z^N(x,t-\tau)\ u(x,\tau)\ dx\ d\tau,\\
 \int_0^t\int_\Omega -\nabla\cdot(\kappa\nabla u(x,\tau))\ w_z^N(x,t-\tau)\ dx\ d\tau
 =& \int_0^t\int_\Omega \kappa(x)\nabla u(x,\tau)\cdot\nabla w_z^N(x,t-\tau)\ dx\ d\tau \\
 &-\int_0^t \int_{\partial \Omega} \kappa(x)\frac{\partial u}{\n}(x,\tau)w_z^N(x,t-\tau\ dx\ d\tau.
\end{aligned}
\end{equation*}
For the boundary condition of $w_z^N(x)$, we have that
\begin{equation*}
 \mathcal B w_z^N=\sum_{l=1}^N \tilde \xi_l(z) \overline{\tilde \xi_l(x)},\quad x\in\partial\Omega.
\end{equation*}
Then it holds that
\begin{align*}
 &\int_0^t \int_{\partial \Omega} \kappa(x)\frac{\partial u}{\n}(x,\tau)w_z^N(x)\ dx\ d\tau\\
 &=\int_0^t \sum_{l=1}^N \tilde\xi_l(z)\ \big\l\frac{\partial u}{\n}(\cdot,\tau),\tilde\xi_l(\cdot) \big\rp \ d\tau-\int_0^t \int_{\partial \Omega} \beta\kappa(x)\frac{\partial u}{\n}(x,\tau)\frac{\partial w_z^N}{\n}(x,t-\tau)\ dx\ d\tau\\
 &=\int_0^t \sum_{l=1}^N \tilde\xi_l(z)\ \big\l\frac{\partial u}{\n}(\cdot,\tau),\tilde\xi_l(\cdot) \big\rp \ d\tau+\int_0^t \int_{\partial \Omega} \kappa(x)u(x,\tau)\frac{\partial w_z^N}{\n}(x,t-\tau)\ dx\ d\tau,
\end{align*}
where the boundary condition $\mathcal B u=0$ is used.
Hence we have
 \begin{align*}
  &\int_0^t \int_\Omega p(x) w_z^N(x,t-\tau)\ dx\ d\tau\\
  &=\int_0^t\int_\Omega (c^{-1}\partial_t+q) w_z^N(x,t-\tau)\ u(x,\tau)
  +\kappa(x)\nabla w_z^N(x,t-\tau)\cdot\nabla u(x,\tau) \ dx\ d\tau\\
  &\quad +\int_0^t \int_{\partial \Omega} \kappa(x)u(x,\tau)\frac{\partial w_z^N}{\n}(x,t-\tau)\ dx\ d\tau -\int_0^t \sum_{l=1}^N \tilde\xi_l(z)\ \big\l\frac{\partial u}{\n}(\cdot,\tau),\tilde\xi_l(\cdot) \big\rp \ d\tau\\
 &=-\int_0^t \sum_{l=1}^N \tilde\xi_l(z)\ \big\l\frac{\partial u}{\n}(\cdot,\tau),\tilde\xi_l(\cdot) \big\rp \ d\tau.
 \end{align*}

 For the above equality, differentiating at time $t$ gives that for $t\in (0,\epsilon)$,
 \begin{equation*}
  -\sum_{l=1}^N \tilde\xi_l(z)\ \big\l\frac{\partial u}{\n}(\cdot,t),\tilde\xi_l(\cdot) \big\rp =\int_\Omega p(x) w_z^N(x,t)\ dx.
 \end{equation*}

From Assumption \ref{condition_f_regularity} $(c)$ we have for a.e. $t\in(0,\epsilon)$,
\begin{equation*}
  -\frac{\partial u}{\n}(z,t)=\lim_{N\to\infty}\int_\Omega p(x) w_z^N(x,t)\ dx.
 \end{equation*}
 The proof is complete.
\end{proof}

With the coefficients $\{c_{z,n}\}$ defined in Lemma \ref{c_z,n}, we can deduce the next corollary straightforwardly.
\begin{corollary}\label{data_formula}
For a.e. $z\in\partial \Omega$ and $t\in(0,\epsilon)$, it holds that
\begin{equation*}
 -\frac{\partial u}{\n}(z, t)=\sum_{n=1}^\infty c_{z,n}p_n (1-e^{-\lambda_n\tau}),
 \end{equation*}
 where $u$ is the solution of equation \eqref{PDE2} and $p_n:=\l p(\cdot),\varphi_n(\cdot)\ro$.
\end{corollary}
\begin{proof}
In this proof, let us the limit $\lim_{N\to \infty}\l  p(\cdot), \overline{w_z^N(\cdot,t)}\ro$.
 Fixing $N\in \mathbb{N}^+$, from the definition of $\psi_z^N$ we have
 $\psi_z^N\in L^2(\Omega)$. So the Fourier expansion of $\psi_z^N$ can be given as $\psi_z^N=\sum_{n=1}^\infty c_{z,n}^N \overline{\varphi_n}$, where $c_{z,n}^N$ is defined in the proof of Lemma \ref{c_z,n}. Moreover, from $w_z^N=u_z^N+\psi_z^N$ and equation \eqref{u_z^N}, we have
 \begin{equation*}
  w_z^N(x,t)=\sum_{n=1}^\infty c_{z,n}^N (1-e^{-\lambda_nt})\overline{\varphi_n(x)}.
 \end{equation*}
 The above result together with the regularity $\psi_z^N\in L^2(\Omega)$ yields that $w_z^N(\cdot,t)\in L^2(\Omega)$ for $t\in [0,\infty)$. Also recall that $p\in L^2(\Omega)$, then
 \begin{equation*}
 \l p(\cdot), \overline{w_z^N(\cdot,t)}\ro
  = \sum_{n=1}^\infty c_{z,n}^N p_n (1-e^{-\lambda_nt}).
 \end{equation*}
 The condition in Assumption \ref{condition_f_regularity} implies that the Dominated Convergence Theorem can be applied to the above series. So we have
\begin{equation*}
 \lim_{N\to\infty}\sum_{n=1}^\infty c_{z,n}^N p_n (1-e^{-\lambda_nt})= \sum_{n=1}^\infty \lim_{N\to\infty}c_{z,n}^N p_n (1-e^{-\lambda_nt}).
\end{equation*}
 From \eqref{c_z,n}, we have
 \begin{equation*}
 \lim_{N\to\infty}\l p(\cdot), \overline{w_z^N(\cdot,t)}\ro= \lim_{N\to\infty}\sum_{n=1}^\infty c_{z,n}^N p_n (1-e^{-\lambda_nt})=\sum_{n=1}^\infty c_{z,n} p_n (1-e^{-\lambda_nt}).
 \end{equation*}
The proof is complete.
\end{proof}

The next two lemmas will be used in the proof of the uniqueness.
\begin{lemma}\label{lemma_uniqueness_1}
Recall that $\{\lambda_j\}_{j=1}^\infty$ is the set of distinct eigenvalues with increasing order, and the coefficients $\{c_{z,n}\}$ are defined in Lemma \ref{c_z,n}.
For any nonempty open subset $\Gamma\subset\partial\Omega$, if $$\sum_{\lambda_n=\lambda_j} c_{z,n}\eta_n=0\ \text{for}\  j\in\mathbb{N}^+\ \text{and\ a.e.}\  z\in\Gamma,$$
then $\{\eta_n\}_{n=1}^\infty=\{0\}$.
\end{lemma}
\begin{proof}
Fixing $j\in \mathbb N^+$, from the proof of Lemma \ref{c_z,n}, we have
 \begin{equation*}
 \begin{aligned}
  c_{z,n}&=\lambda_j^{-1}\lim_{N\to \infty}  \l \psi_z^N,\A\overline{\varphi_n}\ro\\
  &=\lambda_j^{-1}\lim_{N\to \infty}\Big(\l \A\psi_z^N,\overline{\varphi_n}\ro-\l\psi_z^N,\frac{\partial\overline{\varphi_n}}{\n}\rp+\l\frac{\partial\psi_z^N}{\n},\overline{\varphi_n}\rp\Big)\\
  &=-\lambda_j^{-1}\sum_{l=1}^\infty \tilde\xi_l(z)\l\frac{\partial\varphi_n}{\n},\tilde\xi_l\rp.
  \end{aligned}
 \end{equation*}
With the definition of the orthonormal basis $\{\tilde \xi_l\}_{l=1}^\infty$, the above series is actually finite. This gives that for a.e. $z\in\Gamma$,
\begin{equation*}
 \sum_{l=1}^\infty \tilde\xi_l(z)\l \frac{\partial\varphi_n}{\n},\tilde\xi_l\rp=\frac{\partial\varphi_n}{\n}(z),
\end{equation*}
which leads to $\sum_{\lambda_n=\lambda_j}\eta_n\frac{\partial\varphi_n}{\n}(z)=0$ for a.e. $z\in\Gamma$. Assume that $\{\eta_n:\lambda_n=\lambda_j\}\ne\{0\}$, then the linear independence of the set $\{\varphi_n(x):\lambda_n=\lambda_j\}$ yields that $\sum_{\lambda_n=\lambda_j}\eta_n\varphi_n(x)$ is not vanishing on $\Omega$. However, we see that $\sum_{\lambda_n=\lambda_j}\eta_n\varphi_n(x)$ is an eigenfunction corresponding to the eigenvalue $\lambda_j$. These and Lemma \ref{nonempty_open} give that $\sum_{\lambda_n=\lambda_j}\eta_n\frac{\partial\varphi_n}{\n}$ can not vanish almost everywhere on $\Gamma$, which is a contradiction. Hence, we prove that $\{\eta_n:\lambda_n=\lambda_j\}=\{0\}$ for each $j\in\mathbb N^+$, namely, $\eta_n=0$ for $n\in\mathbb N^+$. The proof is complete.
\end{proof}

\begin{lemma}\label{lemma_uniqueness_2}
Set $\Gamma\subset\partial\Omega$ be a nonempty open subset and let $\sum_{n=1}^\infty c_{z,n} \eta_n$ be absolute convergent for a.e. $z\in \Gamma$. Then the result
$$\sum_{n=1}^\infty c_{z,n} \eta_n(1-e^{-\lambda_nt})=0\ \text{for}\ t\in(0,\epsilon)\ \text{and\ a.e.}\ z\in \Gamma$$
implies that $\eta_n=0$ for $n\in\mathbb N^+$.
\end{lemma}
\begin{proof}
From \cite[Lemma 3.5]{RundellZhang:2020} we have that $\sum_{\lambda_n=\lambda_j} c_{z,n} \eta_n=0$ for each $j\in\mathbb N^+$ and a.e. $z\in\Gamma$. Then Lemma \ref{lemma_uniqueness_1} gives the desired result and completes the proof.
\end{proof}

The next proposition concerns with the uniqueness result of equation \eqref{PDE2}.
\begin{proposition}\label{proposition_uniqueness}
  For the PDE \eqref{PDE2}, the data $\frac{\partial u}{\n}\big|_{\Gamma\times(0,\epsilon)}$ can uniquely determine $p(x)$.
 \end{proposition}
 \begin{proof}
  Given $p(x)$ and $\tilde p(x)$, we set $u$ and $\tilde u$ be the corresponding solutions of equation \eqref{PDE2} with $p$ and $\tilde p$ respectively. If $\frac{\partial u}{\n}\big|_{\Gamma\times(0,\epsilon)}=\frac{\partial \tilde u}{\n}\big|_{\Gamma\times(0,\epsilon)},$ we need to show that $\|p-\tilde p\|_{L^2(\Omega)}=0$.

  From Corollary \ref{data_formula}, we have
  \begin{equation*}
 \sum_{n=1}^\infty c_{z,n}(p_n-\tilde p_n) (1-e^{-\lambda_nt})=0,\quad t\in(0,\epsilon).
  \end{equation*}
Then by Lemma \ref{lemma_uniqueness_2}, we conclude that $p_n=\tilde p_n$ for $n\in \mathbb N^+$. This leads to the desired result and completes the proof.
 \end{proof}

 Now we can state the uniqueness theorem of the inverse problem \eqref{IP}.
 \begin{theorem}\label{theorem-uniq}
  For equation \eqref{emission-um}, the data \eqref{data} can uniquely determine the source \eqref{ua-time}.

  More precisely, given two sets of unknown $\{p_k(x)\}_{k=1}^K$ and $\{\tilde p_k(x)\}_{k=1}^K$, we denote the corresponding solutions of equation \eqref{emission-um} by $u_m$ and $\tilde u_m$ respectively. Providing
  $$\frac{\partial u_m}{\n}\big|_{\Gamma\times (0,T)}=\frac{\partial \tilde u_m}{\n}\big|_{\Gamma\times (0,T)},$$
  it holds that
  $$p_k=\tilde p_k\ \text{in}\ L^2(\Omega)\ \text{for}\ k=1,\cdots,K.$$
 \end{theorem}
\begin{proof}
 From equation $\eqref{emission-um}$, if we restrict $t^*\in (0,t_1)$, then we can give the models of $u(x,t^*)$ and $\tilde u(x,t^*)$ as follows:
 \begin{equation*}
 \begin{cases}
  \begin{aligned}
  (c^{-1}\partial_t+\A)u(x,t^*)&=p_1(x),&&  (x,t^*)\in \Omega\times(0,t_1),\\
   u(x,t^*)&=0, &&(x,t^*)\in \partial\Omega\times(0,t_1)\cup \Omega\times\{0\},
  \end{aligned}
 \end{cases}
 \end{equation*}
 and
  \begin{equation*}
 \begin{cases}
  \begin{aligned}
  (c^{-1}\partial_t+\A)\tilde u(x,t^*)&=\tilde p_1(x),&&  (x,t^*)\in \Omega\times(0,t_1),\\
   \tilde u(x,t^*)&=0, &&(x,t^*)\in \partial\Omega\times(0,t_1)\cup \Omega\times\{0\}.
  \end{aligned}
 \end{cases}
 \end{equation*}
 From Proposition \ref{proposition_uniqueness}, we have that $p_1=\tilde p_1$ in the sense of $L^2(\Omega)$, which also gives that $u(x,t_1)=\tilde u(x,t_1)$.

 After that, by setting $t^*=t-t_1$ and $t^*\in(0,t_2-t_1)$, we have
 \begin{equation*}
 \begin{cases}
  \begin{aligned}
  (c^{-1}\partial_t+\A)u(x,t^*)&=p_2(x),&&  (x,t^*)\in \Omega\times(0,t_2-t_1),\\
   u(x,0)&=u(x,t_1), &&x\in \Omega,\\
   u(x,t^*)&=0, &&(x,t^*)\in \partial\Omega\times(0,t_2-t_1),
  \end{aligned}
 \end{cases}
 \end{equation*}
 and
  \begin{equation*}
 \begin{cases}
  \begin{aligned}
  (c^{-1}\partial_t+\A)\tilde u(x,t^*)&=\tilde p_1(x),&&  (x,t^*)\in \Omega\times(0,t_2-t_1),\\
   \tilde u(x,0)&=u(x,t_1), &&x\in\Omega,\\
   u(x,t^*)&=0, &&(x,t^*)\in \partial\Omega\times(0,t_2-t_1).
  \end{aligned}
 \end{cases}
 \end{equation*}
 Next, if we set $w=u-u(x,t_1)$ and $\tilde w=\tilde u-u(x,t_1)$, we have that
  \begin{equation*}
 \begin{cases}
  \begin{aligned}
  (c^{-1}\partial_t+\A)w(x,t^*)&=p_2(x)-\A u(x,t_1),&&  (x,t^*)\in \Omega\times(0,t_2-t_1),\\
  w(x,t^*)&=0, &&(x,t^*)\in \partial\Omega\times(0,t_2-t_1)\cup \Omega\times\{0\},
  \end{aligned}
 \end{cases}
 \end{equation*}
 and
  \begin{equation*}
 \begin{cases}
  \begin{aligned}
  (c^{-1}\partial_t+\A)\tilde w(x,t^*)&=\tilde p_2(x)-\A u(x,t_1),&&(x,t^*)\in \Omega\times(0,t_2-t_1),\\\tilde w(x,t^*)&=0, &&(x,t^*)\in \partial\Omega\times(0,t_2-t_1)\cup \Omega\times\{0\}.
  \end{aligned}
 \end{cases}
 \end{equation*}
 From Proposition \ref{proposition_uniqueness},  we have that $p_2=\tilde p_2$ and $u(x,t_2)=\tilde u(x,t_2)$.

 Continuing this argument, we can achieve the desired result and complete the proof.
\end{proof}

\subsection{Conditional stability of Lipschitz type}
We next consider the Lipschitz stability for our inverse problem of recovering $\{p_k(x)\}_{k=1}^K$ in \eqref{ua-time} from the boundary data \eqref{data}. Similar as done in \cite{Sun2020AMAS, Sun2020IP}, the crucial point is to construct an integral identity connecting the inversion input data with the unknown functions.

We introduce the adjoint problem of \eqref{emission-um} as
\begin{equation}\label{adjoint-phi}
\begin{cases}
\begin{aligned}
 (-c^{-1}\partial_t +\A)\phi(x,t)&=0, \qquad\qquad x\in\Omega, \ t\in(0,T),\\
 \phi(x,T)&= 0,  \qquad\qquad x\in \Omega,\\
{\mathcal B} {\phi} &=
\begin{cases}
\omega(x,t),  & x\in\Gamma, \ t\in (0,T),\\
0, & x\in\partial\Omega\setminus\Gamma, \ t\in (0,T).
\end{cases}\\
\end{aligned}
\end{cases}
\end{equation}
Then the following lemma gives a variational identity.
%\begin{remark}
%We take the transform $\psi(x,t)=\phi(x,T-t)$,
%then the smooth function $\psi(x,t)$ meets
%\begin{equation}\label{adjoint-psi}
%\begin{cases}
%\begin{aligned}
% \partial_t \psi + \A{\psi} &=0, \qquad\qquad\qquad x\in\Omega, \, t\in(0,T),\\
% {\psi}(x,0)&= 0,  \qquad\qquad\qquad x\in \Omega,\\
%{\mathcal B} \psi &=
%\begin{cases}
%\omega(x,T-t),  & x\in\Gamma, \, t\in (0,T),\\
%0, & x\in\partial\Omega\setminus\Gamma, \, t\in (0,T).
%\end{cases}
%\end{aligned}
%\end{cases}
%\end{equation}
%\end{remark}

\begin{lemma}\label{thm3-1}
Let $\Gamma$ be a nonempty open subset of boundary $\partial\Omega$. For the given time mesh $\{t_k\}_{k=0}^K$ with $t_K=T$, we denote $u_m$ and $\tilde u_m$ as the solution of equation \eqref{emission-um} with
$$S[\mu_f,u_e]=\sum_{k=1}^K p_k(x)\chi_{{}_{t\in [t_{k-1},t_k)}}\ \text{and}\ \tilde S[\tilde \mu_f,u_e]=\sum_{k=1}^K \tilde p_k(x)\chi_{{}_{t\in [t_{k-1},t_k)}}$$
respectively. Then there holds that
\begin{equation}\label{iden-obv}
\begin{aligned}
\sum_{k=1}^K \int_{t_{k-1}}^{t_k}\int_{\Omega} (p_k-\tilde p_k)(x)\ \phi[\omega](x,t) \ dx\ dt = \int_{0}^T\int_{\Gamma} \omega(x,t) \big(\frac{\partial u_m}{\n} - \frac{\partial \tilde u_m}{\n} \big)\ dx\ dt,
\end{aligned}
\end{equation}
where $\phi[\omega]$ is the solution to \eqref{adjoint-phi} for $\omega\in L^2(\Gamma\times(0,T))$.
\end{lemma}

\begin{proof}
We see that $U_m:=u_m-\tilde u_m$ satisfies
\begin{equation}\label{emission-UM}
\begin{cases}
\begin{aligned}
\left(c^{-1}\partial_t +\A \right) U_m(x,t)&=\sum_{k=1}^K (p_k-\tilde p_k)(x)\chi_{{}_{t\in [t_{k-1},t_k)}}, && (x,t)\in\Omega\times (0,T),\\
U_m(x,0)&=0, && x\in\Omega,\\
\mathcal B U_m&=0, && (x,t)\in\partial\Omega\times(0,T).
\end{aligned}
\end{cases}
\end{equation}

Multiplying $\phi(x,t):=\phi[\omega](x,t)$ on the two sides of equation \eqref{emission-UM} and integrating them in $\Omega\times(0,T)$, we get
\begin{eqnarray}\label{sun3-5}
\int_{0}^T\int_{\Omega} (c^{-1}\partial_t +\mathcal{A} )U_m(x,t)\phi(x,t) \ dx\ dt=
\sum_{k=1}^K \int_{t_{k-1}}^{t_k}\int_{\Omega} (p_k-\tilde p_k)(x) \phi[\omega](x,t) \ dx\ dt.
\end{eqnarray}

By $\phi(x,T)=0$, $U_m(x,0)=0$ and $\{\phi, U_m\}\subset L^2 (0,T;H^2(\Omega))$, we have
\begin{equation}
\begin{aligned}
\int_{0}^T\int_{\Omega}\ (c^{-1}\partial_t U_m) \  \phi\  dx\ dt &={\int_{\Omega} (U_m \phi)|_{t=0}^{t=T}\ dx}-\int_{0}^T\int_{\Omega} U_m (c^{-1}\partial_t \phi) \ dx\ dt \\
&= -\int_{0}^T\int_{\Omega} U_m (c^{-1}\partial_t \phi) \ dx\ dt.
\end{aligned}
\end{equation}
The Green formula also gives that
\begin{equation}\label{sun3-7}
\begin{aligned}
\int_{0}^T\int_{\Omega} (\A U_m)\ \phi\ dx\ dt
&=\int_{0}^T\int_{\Omega} (\mathcal{A}{\phi})\ U_m \ dx\ dt \\
&\quad + \int_{0}^T\int_{\partial\Omega} \big(\phi\frac{\partial U_m}{\n} -U_m\frac{\partial \phi}{\n}\big) \ dx\ dt.
\end{aligned}
\end{equation}
Combining \eqref{sun3-5}-\eqref{sun3-7} together and using \eqref{adjoint-phi}, we finally have
\begin{equation*}\label{sun3-8}
\begin{aligned}
\sum_{k=1}^K \int_{t_{k-1}}^{t_k}\int_{\Omega} (p_k-\tilde p_k)(x) \phi[\omega](x,t) \ dx\ dt
&= \int_{0}^T \int_{\partial\Omega} \big(\phi\frac{\partial U_m}{\n} -U_m\frac{\partial \phi}{\n}\big) \ dx\ dt \\
&= \int_{0}^T \int_{\partial\Omega} (\mathcal{B}\phi) \frac{\partial U_m}{\n} \ dx\ dt.
\end{aligned}
\end{equation*}
Henceforth, utilizing the boundary condition
\begin{equation*}
\mathcal{B}\phi =
\begin{cases}
\omega(x,t),   &x\in\Gamma, \ t\in (0,T),\\
0,  &x\in\partial\Omega\setminus\Gamma, \ t\in (0,T),
\end{cases}\\
\end{equation*}
we deduce the desired result and complete the proof.
\end{proof}

For the set of functions $\{p_k(x)\}_{k=1}^K$, we use the related vector form as
$\vec p:=(p_1,\cdots,p_K)$ and define the operator $L:(L^2(\Omega))^K\to L^2(\Omega\times(0,T))$ as
$$L\vec p=\sum_{k=1}^K p_k(x)\chi_{{}_{t\in[t_{k-1},t_k)}}.$$

Next we give a bilinear functional with respect to $L\vec p$ and $\omega(x,t)$ by
\begin{equation}\label{B-func}
\mathcal{L} (L \vec p,\omega)= \int_{\Omega\times(0,T)} (L\vec p)(x,t)\ \phi[\omega](x,t) \ dx\ dt,
\end{equation}
where $\phi[\omega](x,t)$ satisfies \eqref{adjoint-phi}. The corresponding norm $\|\cdot\|_{\mathcal L}$ is introduced in the following lemma.

\begin{lemma}
With \eqref{B-func}, we define the function $\|\cdot\|_\mathcal L:  l^1((L^2(\Omega))^K)\to \mathbb R$ as
\begin{eqnarray}\label{sun4-2}
\|\vec p\|_\mathcal{L}:=\mathop {\sup}\limits_{\omega\in {W} }\frac{|\mathcal{L} (L \vec p,\omega)|}{\norm{\omega}_{L^2(\Gamma\times(0,T))}},
\end{eqnarray}
where
$$W:=\{\psi:  \psi\in L^2(\Gamma\times(0,T)),\ \psi\not\equiv 0\}.$$
Then $\|\cdot\|_\mathcal L$ is a norm on $l^1((L^2(\Omega))^K)$.
\end{lemma}

\begin{proof}
Firstly, let us prove the well-definedness of $\|\vec p\|_{\mathcal L}$. With Holder inequality, we have
\begin{equation*}
\begin{aligned}
|\mathcal{L} (L \vec p,\omega)| &\le \sum_{k=1}^K \int_{t_{k-1}}^{t_k} \int_\Omega \big|p_k(x)\ \phi[\omega](x,t)\big|\ dx\ dt\\
 &\le \sum_{k=1}^K \norm{p_k}_{L^2(\Omega)}\norm{\phi[\omega]}_{L^2(\Omega\times(t_{k-1},t_k))}
\le C\norm{\omega}_{L^2(0,T;\Gamma)} \sum_{k=1}^K \norm{p_k}_{L^2(\Omega)}\\
&\le C \sum_{k=1}^K \norm{\omega}_{L^2(t_{k-1},t_k;\Gamma)} \norm{p_k}_{L^2(\Omega)}
\le C  \norm{\omega}_{L^2(0,T;\Gamma)} \sum_{k=1}^K\norm{p_k}_{L^2(\Omega)},
\end{aligned}
\end{equation*}
which gives
$$\frac{|\mathcal{L} (L \vec p,\omega)|}{\norm{\omega}_{L^2(\Gamma\times(0,T))}} \le C\sum_{k=1}^K \norm{p_k}_{L^2(\Omega)}=C\|\vec p\|_{l^1((L^2(\Omega))^K)},\ \omega\in W.$$
So we have
$\|\vec p\|_{\mathcal L}\le C\|\vec p\|_{l^1((L^2(\Omega))^K)}<\infty$.

Next we prove that $\|\cdot\|_{\mathcal L}$ is a norm on $ l^1((L^2(\Omega))^K)$. Firstly, it is obviously that
$$\|\vec p\|_{\mathcal L}\ge 0\ \text{and}\ \|c\vec p\|_{\mathcal L}=|c|\|\vec p\|_{\mathcal L}\ \text{for}\ c\in \mathbb R\ \text{and}\ \vec p\in l^1((L^2(\Omega))^K).$$
Secondly, from the definition of $\|\cdot\|_{\mathcal L}$, we can prove that
\begin{align*}
 \|\vec{p_1}+\vec{p_2}\|_{\mathcal L}&=\mathop {\sup}\limits_{\omega\in {W} }\frac{|\mathcal{L} (L \vec p_1+L\vec p_2,\omega)|}{\norm{\omega}_{L^2(\Gamma\times(0,T))}}
\le \mathop {\sup}\limits_{\omega\in {W} }\frac{|\mathcal{L} (L \vec p_1,\omega)|+|\mathcal{L} (L\vec p_2,\omega)|}{\norm{\omega}_{L^2(\Gamma\times(0,T))}}\\
&\le \mathop {\sup}\limits_{\omega\in {W} }\frac{|\mathcal{L} (L \vec p_1,\omega)|}{\norm{\omega}_{L^2(\Gamma\times(0,T))}}+\mathop {\sup}\limits_{\omega\in {W} }\frac{|\mathcal{L} (L\vec p_2,\omega)|}{\norm{\omega}_{L^2(\Gamma\times(0,T))}}=\|\vec p_1\|_{\mathcal L}+\|\vec p_2\|_{\mathcal L},
\end{align*}
which is the triangle inequality.
At last, we need to show that $\|\vec p\|_\mathcal{L}=0$ leads to $\vec p=0$.
Given $\|\vec p\|_\mathcal{L}=0$, we have that $|\mathcal{L} (L \vec p,\omega)|=0$ for each $\omega\in W$. This together with Lemma \ref{thm3-1} yields that
\begin{equation*}\label{sun4-4}
 \int_{0}^T\int_{\Gamma} \omega(x,t) \frac{\partial u_m}{\n} \ dx\ dt=0, \quad \forall \omega\in L^2(\Gamma\times(0,T)),
\end{equation*}
which leads to $\frac{\partial u_m}{\n}=0$ in $L^2(\Gamma\times(0,T))$. Then by Theorem \ref{theorem-uniq}, we deduce that $p_k=0$ in $L^2(\Omega)$ for $k=1,\cdots,K$, i.e., $\vec p=0$.
Now we have proved that $\|\cdot\|_{\mathcal L}$ is a norm on $l^1((L^2(\Omega))^K)$ and the proof is complete.
\end{proof}

Now we can establish the conditional stability of Lipschitz type for the inverse problem by the weighted norm $\norm{\cdot}_\mathcal{L}$.

\begin{theorem}\label{stability-inv}
We set $\Gamma$ to be a nonempty open subset of boundary $\partial\Omega$, and denote the solutions of \eqref{emission-um} with $\vec p=(p_1,\cdots,p_K)$ and $\vec{\tilde p}=(\tilde p_1,\cdots,\tilde p_K)$ by $u_m$ and $\tilde u_m$ respectively. Then the next stability result holds
\begin{eqnarray*}\label{stability}
\|\vec p-\vec {\tilde p}\|_{\mathcal{L}} \le \Big\|\frac{\partial u_m}{\n}-\frac{\partial {\tilde u}_m}{\n}\Big\|_{L^2(\Gamma\times(0,T))}.
\end{eqnarray*}
\end{theorem}

\begin{proof}
By \eqref{iden-obv} and the definition (\ref{sun4-2}), we have
\begin{equation*}
\begin{aligned}
\|\vec p-\vec {\tilde p}\|_{\mathcal{L}}&=\mathop {\sup}\limits_{\omega\in {W} }\frac{|\mathcal L(L(\vec p-\vec{\tilde p}),\omega)|}{\norm{\omega}_{L^2(\Gamma\times(0,T))}}\\
&=\mathop {\sup}\limits_{\omega\in {W} }\|\omega\|^{-1}_{L^2(\Gamma\times(0,T))}\Big| \int_{0}^T\int_{\Gamma} \omega(x,t) \big(\frac{\partial u_m}{\n} - \frac{\partial \tilde u_m}{\n} \big)\ dx\ dt\Big|.
\end{aligned}
\end{equation*}
With Cauchy-Schwartz inequality, we immediately deduce that
\begin{equation*}\label{sun4-9}
\|\vec p-\vec {\tilde p}\|_{\mathcal{L}} \leq  \Big\|\frac{\partial u_m}{\n}-\frac{\partial {\tilde u}_m}{\n}\Big\|_{L^2(\Gamma\times(0,T))}.
\end{equation*}
The proof is complete.
\end{proof}

%\begin{remark}
%\red{The norms $\norm{\cdot}_\mathcal{L}$ and
%$\norm{\cdot}_{L^2(\Omega)}$ are not equivalent. Theorem \ref{stability-inv} can be considered as the conditional stability for our inverse problem in $L^2(\Omega)$ by some weighted norm $\norm{\cdot}_\mathcal{L}$, which is a little bit weaker than the standard $L^2$ norm.}
%\end{remark}

\section{Generalization error estimates.}\label{sec_pinn}
In this section, we introduce the proposed reconstruction scheme with its associated error estimates. This lays the foundations of our approach that we subsequently test numerically. We decompose this inverse problem \eqref{IP} into solving a forward problem to get $u_e$ and an inverse problem for recovering $(\mu_f, u_m)$. More precisely, we solve the forward problem \eqref{excitation-Ue} with the known boundary input to obtain $u_e(x,t)$, and then reconstruct $\mu_f$ and $u_m$ from the observation data \eqref{data} with the obtained solution $u_e$.

\subsection{Loss functions.}\label{Sec2-1}

We assume that the activation function is of $C^2$ regularity. This includes activation such as $\sigma=\tanh$.
Then, for the neural network $u_{e,\theta_e}$, $u_{m,\theta_m}$ defined by \eqref{DeepNN} in terms of $\sigma$, we set $u_{e,\theta_e}, u_{m,\theta_m} \in C^{l}(\overline\Omega\times [0, T])$ for
$l=0,1,2$. For the network parameters
$\theta_e, \theta_m\in\Theta:=\{\{(W_{k}, b_{k})\}_{k=1}^{K} :W_{k} \in \mathbb{R}^{d_{k} \times d_{k-1}}, b_{k} \in \mathbb{R}^{d_{k}}\}$,
the set of all possible trainable parameters, $u_{e,\theta_e}(x,t)$, $u_{m,\theta_m}(x,t)$ up to second order derivatives are bounded in $\overline\Omega\times [0,T]$ for any specified $\theta_e, \theta_m$.

For given boundary input, it is necessary to solve the forward problem 
\eqref{excitation-Ue} first. We denote the approximate solution of \eqref{excitation-Ue} is $u_{e,\theta_e}$, which is a deep neural networks function with the networks parameters $\theta_e$. Next the following residuals for the forward problem \eqref{excitation-Ue} are introduced.
\begin{itemize}
\item Interior PDE residual
\begin{eqnarray*}\label{res1-0}
\mathcal{R}_{{int},\theta_e}(x, t):= \left(c^{-1}\partial_t +\A \right) u_{e,\theta_e}(x,t), \quad (x,t) \in \Omega\times (0,T).
\end{eqnarray*}
\item Spatial boundary residual
\begin{eqnarray*}\label{res2-0}
\mathcal{R}_{sb,\theta_e}(x, t):= \mathcal B u_{e,\theta_e}(x, t)-\mathcal B u_{e}(x, t), \quad (x,t) \in \partial\Omega \times (0,T).
\end{eqnarray*}
\item Temporal boundary (initial status) residual
\begin{eqnarray*}\label{res3-0}
\mathcal{R}_{tb, \theta_e}(x):=u_{e,\theta_e}(x, 0), \quad  x \in \Omega.
\end{eqnarray*}
\end{itemize}
We minimize the following loss function
\begin{eqnarray}\label{loss-ue}
J_1(\theta_e) = \| \mathcal{R}_{{int},\theta_e}\|_{L^2(0,T;L^{2}(\Omega))}
+ \| \mathcal{R}_{tb,\theta_e}\|_{L^2(\Omega)}
+\|\mathcal{R}_{sb,\theta_e}\|_{L^2(0,T;L^{2}(\partial\Omega))}
\end{eqnarray}
to search for the optimal parameters $\theta^*_e$.
The obtained $u_e^*(x,t):=u_{e,\theta_e^{*}}(x,t)$ can be regarded as approximate solution of the forward problem \eqref{excitation-Ue}.

Next, we introduce a form loss function of data-driven solution of inverse problems
that ensures reconstruction accuracy owing to the conditional stability of the inverse problem.
For this purpose, we need to define suitable residuals measuring the errors of the governed system and the input data.
Define the following residuals for the  emission system.
\begin{itemize}
\item Interior PDE residual
\begin{eqnarray*}\label{res1}
\mathcal{R}_{{int},\theta_m,\theta_f,\theta_e^*}(x, t):= (c^{-1}\partial_t +\A ) u_{m,\theta_m}(x,t)- \mu_{f,\theta_f}(x,t)u_{e,\theta_e^*}(x, t; x_s), \quad (x,t) \in \Omega\times (0,T).
\end{eqnarray*}
\item Spatial boundary residual
\begin{eqnarray*}\label{res2}
\mathcal{R}_{sb,\theta_m}(x, t):= \mathcal B u_{m,\theta_m}(x, t), \quad (x,t) \in \partial\Omega\times (0,T).
\end{eqnarray*}
\item Temporal boundary (initial status) residual
\begin{eqnarray*}\label{res3}
\mathcal{R}_{tb, \theta_m}(x):=u_{m,\theta_m}(x, 0), \quad  x \in \Omega.
\end{eqnarray*}
\item Data residual
\begin{eqnarray*}\label{res4}
%\mathcal{R}_{d, \theta}(x):=u_{\theta}(x, T)-u(x,T), \quad  x \in \Omega.
\mathcal{R}_{d, \theta_m}(x,t):=\frac{\partial u_{m,\theta_m}}{\n}-\varphi^\delta(x,t), \quad  (x,t) \in \Gamma\times(0,T).
\end{eqnarray*}
\end{itemize}
Thus, a loss function minimization scheme seeks to minimize these residuals comprehensively with some weights balancing different residuals. We define a new loss function involving the norm of the derivatives for some residuals, namely
\begin{equation}\label{loss-um}
\begin{aligned}
 J_2(\theta_f, \theta_m)
 =& \lambda\| \mathcal{R}_{d, \theta_m}\|_{L^2(\Gamma\times(0,T))}
+\|\mathcal{R}_{{int},\theta_m,\theta_f,\theta_e^*}\|_{L^2(0,T;L^{2}(\Omega))}\\
%&& \|\mathcal{R}_{sb,\theta_m^*}\|_{L^2(0,T;{H^{1/2}(\partial\Omega)})}
%+ \|\partial_t \mathcal{R}_{sb,\theta_m^*}\|_{L^2(0,T;{H^{1/2}(\partial\Omega)})}
&+\| \mathcal{R}_{tb,\theta_m}\|_{H_0^1(\Omega)}+\|\mathcal{R}_{sb,\theta_m}\|_{H^1(0,T;{H^{1}(\partial\Omega)})},
\end{aligned}
\end{equation}
where $\lambda$ is a hyper-parameter to balance the residuals between the knowledge of PDE and the measurement data.
%The reason that we use the loss function in the form of (\ref{Loss1}) with the derivatives penalties on the residuals is that the conditional stability result for recovering $f$ requires higher regularity on the measurement data $u(\cdot,T)$ in Lemma \ref{lem3-0}. In this sense, the loss function design for inverse problems differs from that for forward problems such as PINN. The smoothness requirements are not only to ensure the existence of solutions to forward problems, but also essential to guarantee the well-posedness of the inverse problem under the framework of optimization. We will demonstrate by experiments that regularization scheme in terms of derivative regularity is important for ensuring good numerical performances.

To evaluate the integrals in \eqref{loss-ue} and \eqref{loss-um} numerically, we introduce the training sets
\begin{eqnarray*}
&&\mathcal{S}_{d}:=\left\{(x_n,t_n): (x_n,t_n)\in \Gamma\times(0,T],\quad n=1,2,\cdots,N_{d} \right\},\nonumber \\
&&\mathcal{S}_{int}:=\left\{(\widetilde{x}_{n},\widetilde{t}_{n}): (\widetilde{x}_{n},\widetilde{t}_{n})\in \Omega_T, \quad n=1,2,\cdots,N_{int}\right\}, \nonumber \\
&&\mathcal{S}_{tb}:=\left\{(\overline{x}_{n},0): \overline{x}_{n}\in \Omega, \quad n=1,2,\cdots,N_{tb}\right\},\\
&&\mathcal{S}_{sb}:=\left\{(\widehat{x}_{n},\widehat{t}_{n}): (\widehat{x}_{n},\widehat{t}_{n})\in \partial\Omega_T,\quad n=1,2,\cdots,N_{sb} \right\}.\nonumber
\end{eqnarray*}
Applying these sets and the numerical quadrature rules \cite{Mishra2}, we can consider the following two empirical loss function
\begin{eqnarray}\label{Loss1-1}
J_1^N(\theta_e)
=\sum_{n=1}^{\bar N_{int}}\bar\omega_n^{int}\Big|\mathcal{R}_{{int},\theta_e}(\widetilde{x}_{n}, \widetilde{t}_{n})\Big|^2
+\sum_{n=1}^{\bar N_{tb}}\bar\omega_n^{tb}\Big|\mathcal{R}_{tb,\theta_e}(\overline{x}_{n})\Big|^2
+\sum_{n=1}^{\bar N_{sb}}\bar\omega_n^{sb}\Big|\mathcal{R}_{sb,\theta_e}(\widehat{x}_{n}, \widehat{t}_{n})\Big|^2,
\end{eqnarray}
and
\begin{equation}\label{Loss1-2}
\begin{aligned}
J_2^N(\theta_f,\theta_m)
=&\lambda\sum_{n=1}^{N_{d}}\omega_n^{d}\Big|\mathcal{R}_{d,\theta_m}(x_n,t_n)\Big|^2
+\sum_{n=1}^{N_{int}}\omega_n^{int}\Big|\mathcal{R}_{{int},\theta_f,\theta_m,\theta_e^*}(\widetilde{x}_{n}, \widetilde{t}_{n})\Big|^2 \\
&+\sum_{n=1}^{N_{tb}}\omega_n^{tb,0}\Big|\mathcal{R}_{tb,\theta_m}(\overline{x}_{n})\Big|^2
+\sum_{n=1}^{N_{tb}}\omega_n^{tb,1}\Big|\partial_t\mathcal{R}_{tb,\theta_m}(\overline{x}_{n})\Big|^2\\
&+\sum_{n=1}^{N_{sb}}\omega_n^{sb,0}\Big|\mathcal{R}_{sb,\theta}(\widehat{x}_{n}, \widehat{t}_{n})\Big|^2
+\sum_{n=1}^{N_{sb}}\omega_n^{sb,1}\Big|\frac{\partial\mathcal{R}_{sb,\theta}}{\n}(\widehat{x}_{n}, \widehat{t}_{n})\Big|^2 \\
&+\sum_{n=1}^{N_{sb}}\omega_n^{sb,2}\Big| \frac{\partial \mathcal{R}_{sb,\theta}}{\partial t}(\widehat{x}_{n}, \widehat{t}_{n})\Big|^2
+\sum_{n=1}^{N_{sb}}\omega_n^{sb,3}\Big|\frac{ \partial^2 \mathcal{R}_{sb,\theta}}{\partial t\n}(  \widehat{x}_{n}, \widehat{t}_{n})\Big|^2,
\end{aligned}
\end{equation}
where the coefficients
$\bar\omega^{int}_n,\;\bar\omega^{tb}_n,\; \bar\omega^{sb}_n,\ \omega^{d}_n,\;
\omega^{int}_n,\;\omega^{tb,k}_n,\; \omega^{sb,j}_n$ with $k=0,1$ and $j=0,1,2,3$ are the quadrature weights. Applying the quadrature rules \eqref{quadrature}, it is easy to see that the error for the loss function is
\begin{equation}\label{Loss-error}
\begin{aligned}
|J_1(\theta_e)-J_1^N(\theta_e)|
&\le C\max\{ \bar N_{int}^{-\bar\alpha_{int}},\bar N_{tb}^{-\bar\alpha_{tb}},
\bar N_{sb}^{-\bar\alpha_{sb}}\},\\
|J_2(\theta_f,\theta_m)-J_2^N(\theta_f,\theta_m)|
&\le C\max_{k=0,1,j=0,1,2,3}\{ N_{d}^{-\alpha_{d}},N_{int}^{-\alpha_{int}},N_{tb}^{-\alpha_{tb,k}},
N_{sb}^{-\alpha_{sb,j}}\},
\end{aligned}
\end{equation}
where $C$ depends on the continuous norm $\|\cdot\|_{C(\Omega)}$ of the integrands.
Therefore, the underlying solutions and neural networks have to be sufficiently regular such that the residuals can be approximated to high accuracy by the quadrature rules.

\subsection{The estimates of generalization error.}\label{error}

We give the regularity estimate of the PDEs with non-homogenous boundary condition.
The following $L^p$ estimate on elliptic system can be found in reference \cite{agmon1959estimates}.

\begin{lemma}\label{HtwoM}
Let $w(x)\in H^2(\Omega)$ solve
\begin{equation*}
\begin{cases}
\begin{aligned}
-\Delta w&=f(x), &&x\in\Omega,\\
\frac{\partial w}{\n}+ \beta w&=b(x), &&x\in\partial\Omega,
\end{aligned}
\end{cases}
\end{equation*}
for $0<\beta_0\le \beta(x)\in C(\partial\Omega)$. Then we have
$$\norm{w}_{H^2(\Omega)}\le C(\norm{f}_{L^2(\Omega)}+\|\tilde b\|_{H^1(\Omega)})\le C(\norm{f}_{L^2(\Omega)}+\|b\|_{H^{1/2}(\partial\Omega)}),$$
where $\tilde b(x)\in H^1(\Omega)$ is the extension of $b(x)\in H^{1/2}(\partial\Omega)$.
\end{lemma}

%The extension of $b$ from $\partial\Omega$ to $\Omega$ can be implemented by
%\begin{eqnarray*}
%\begin{cases}
%-\Delta \tilde b(x)=0, &x\in\Omega\\
%\tilde b(x)=b(x), &x\in\partial\Omega,
%\end{cases}
%\end{eqnarray*}
%leading to
%$\|\tilde b\|_{H^1(\Omega)}\le C\|b\|_{H^{1/2}(\partial\Omega)}$.

\begin{lemma}
Consider the following PDEs with non-homogenous boundary condition
\begin{equation}\label{PDEs-regular}
\begin{cases}
\begin{aligned}
\left(c^{-1}\partial_t +\A \right)  V(x,t)&=0,  &&(x,t)\in\Omega\times (0,T),\\
V(x,0)&=0, &&x\in\Omega,\\
\mathcal B V(x,t)&=B(x,t), && (x,t)\in\partial\Omega\times(0,T),
\end{aligned}
\end{cases}
\end{equation}
if the boundary condition $B\in H^1(0,T;{H^{1/2}(\partial\Omega)})$, then there exists a unique solution $V \in L^2(0,T;H^2(\Omega))$ with the estimate
\begin{align*}
\| V \|_{L^2(0,T;H^{2}(\Omega))}
\leq C \|B\|_{H^1(0,T;{H^{1/2}(\partial\Omega)})}.
\end{align*}
\end{lemma}

\begin{proof}
For any fixed $t\in[0,T]$, define $\Lambda_b(x, t)$ from
\begin{eqnarray}\label{elliptic}
\begin{cases}
 -\A  \Lambda_b(x, t)=0, &x\in \Omega,\\
  \mathcal B \Lambda_b(x, t)=B(x,t), &x\in\partial\Omega.
\end{cases}
\end{eqnarray}
By the regularity estimate on elliptic problem and Lemma \ref{HtwoM},  there exists a unique solution $\Lambda_b(\cdot,t)\in H^2(\Omega)$ satisfying $\|\Lambda_b(\cdot,t)\|_{H^{2}(\Omega)}\le C_t\|\tilde B(\cdot, t)\|_{H^1(\Omega)}$. Since $\beta(x)$ is independent of $t$, we also have for $k=0,1,2$ that
\begin{eqnarray*}\label{SEP00-1}
\|\partial_t^k \Lambda_b(\cdot,t)\|_{H^2(\Omega)}\leq C \|\partial_t^k \tilde B(\cdot,t)\|_{H^1(\Omega)}\le C\|\partial_t^k B(\cdot,t)\|_{H^{1/2}(\partial\Omega)},\quad t\in[0,T].
\end{eqnarray*}
Let $AC[0,T]$ be the space of absolutely continuous functions. For $B(x,\cdot)\in AC[0,T]$, it follows $\Lambda_b(x,\cdot)\in AC[0,T]$ by \eqref{elliptic}. So we can decompose $V=V_c+\Lambda_b$, where $V_c$ satisfies
\begin{equation*}\label{FB4-an}
\begin{cases}
\begin{aligned}
(c^{-1}\partial_t +\A ) V_c(x,t)&=-c^{-1} \partial_t \Lambda_b, &&(x,t)\in \Omega\times(0,T),\\
\mathcal{B} V_c(x,t)&=0, &&(x,t)\in \partial\Omega\times [0,T],\\
V_c(x,0)&=-\Lambda_b(x,0)=0, &&x\in\Omega.
\end{aligned}
\end{cases}
\end{equation*}
From the regularity estimates in \cite{Evans}, there holds
\begin{eqnarray*}
\| V_c \|_{L^2(0,T;H^{2}(\Omega))}
\leq C\| \partial_t\Lambda_b\|_{L^2(0,T;L^2(\Omega))}
\leq C\|\partial_t B\|_{L^2(0,T;{H^{1/2}(\partial\Omega)})}.
\end{eqnarray*}
Sequentially, we have
\begin{align*}
\| V \|_{L^2(0,T;H^{2}(\Omega))}
&\leq \| \Lambda_b \|_{L^2(0,T;H^{2}(\Omega))} + \| V_c \|_{L^2(0,T;H^{2}(\Omega))}   \\
&\leq C(\|B\|_{L^2(0,T;{H^{1/2}(\partial\Omega)})} + \|\partial_t B\|_{L^2(0,T;{H^{1/2}(\partial\Omega)})} ) \\
&\leq C \|B\|_{H^1(0,T;{H^{1/2}(\partial\Omega)})}.
\end{align*}
The proof is complete.
\end{proof}
Now, we define the generalization errors as
\begin{equation}\label{gener1}
\begin{cases}
\begin{aligned}
\mathcal{E}_{G,u_e}:=&\left\|u_e^{*}-u_{e,ex}\right\|_{C\left([0, T] ;L^{2}(\Omega)\right)},\\
\mathcal{E}_{G,\vec {p}}:=&\left\| {\vec p_*-\vec p_{ex}}\right\|_{\mathcal{L}}, \\
\mathcal{E}_{G,u_m}:=&\left\|u_m^{*}-u_{m,ex}\right\|_{C\left([0, T] ;
L^{2}(\Omega)\right)},
\end{aligned}
\end{cases}
\end{equation}
where $(u^{*}_e, u^{*}_m, \mu_f^*) :=(u_{e,\theta_e^{*}},u_{m,\theta_m^{*}},\mu_{f,\theta_f^*})$ with the minimizer $(\theta_e^*,\theta_m^*,\theta_f^*)$ of functional \eqref{Loss1-1} and \eqref{Loss1-2}. They can be used to approximate  the exact solution $(u_{e,ex},u_{m,ex},\mu_{f,ex})$ of the inverse problem \eqref{IP}. The related vector forms are given as $\vec p_*:=(p^*_1,\cdots,p^*_K)$ and
$\vec p_{ex}:=(p_{1,ex},\cdots,p_{K,ex})$, which are the approximate networks solution corresponding to $\mu_f^* u_e^*$ and exact $\mu_{f,ex} u_{e,ex}$, respectively. We will estimate the generalization errors in terms of the training errors for both excitation and emission, which are given below.

Define the training errors for excitation process:
\begin{itemize}
\item The interior PDE training errors
\begin{eqnarray*}\label{training1-ue}
\bar{\mathcal{E}}_{T,int}:=\Big(\sum_{n=1}^{\bar N_{int}}\bar\omega_n^{int}\Big|\mathcal{R}_{{int},\theta^*_e}(\widetilde{x}_{n}, \widetilde{t}_{n})\Big|^2\Big)^{1/2}.
\end{eqnarray*}
\item The initial condition training errors
\begin{eqnarray*}\label{training2-ue}
\bar{\mathcal{E}}_{T,tb}
:=\Big(\sum_{n=1}^{\bar N_{tb}}\bar\omega_n^{tb}\Big|\mathcal{R}_{tb,\theta^{*}_e}(\overline{x}_{n})\Big|^2\Big)^{1/2}.
\end{eqnarray*}
\item The spatial boundary condition training errors
\begin{eqnarray*}\label{training3-ue}
\bar{\mathcal{E}}_{T,sb}
:=\Big(\sum_{n=1}^{\bar N_{sb}}\bar \omega_n^{sb}\Big|\mathcal{R}_{sb,\theta^{*}_e}(\widehat{x}_{n}, \widehat{t}_{n})\Big|^2\Big)^{1/2}.
\end{eqnarray*}
\end{itemize}
Define the training errors for emission process:
\begin{itemize}
\item The measurement data training errors
\begin{eqnarray*}\label{training1}
\mathcal{E}_{T,d}:=\Big(\sum_{n=1}^{N_{d}}\omega_n^{d} \Big|\mathcal{R}_{d,\theta^*_m}(x_n,t_n) \Big|^2\Big)^{1/2}.
\end{eqnarray*}
\item The interior PDE training errors
\begin{eqnarray*}\label{training2}
\mathcal{E}_{T,int}:=\Big(\sum_{n=1}^{N_{int}}\omega_n^{int} \Big|\mathcal{R}_{{int},\theta^*_f,\theta^*_m,\theta^*_e}(\widetilde{x}_{n}, \widetilde{t}_{n})\Big|^2\Big)^{1/2}.
\end{eqnarray*}
\item The initial condition training errors $\mathcal{E}_{T,tb,0}:=\mathcal{E}_{T,tb}+\mathcal{E}_{T,tb,1}$, where
\begin{eqnarray*}\label{training4}
\begin{cases}
\mathcal{E}_{T,tb,0}
:=\left(\sum_{n=1}^{N_{tb}}\omega_n^{tb,0}\Big|\mathcal{R}_{tb,\theta^{*}_m}(\overline{x}_{n})\Big|^2\right)^{1/2},
\\
\mathcal{E}_{T,tb,1}
:=\left(\sum_{n=1}^{N_{tb}}\omega_n^{tb,1} \Big|\partial_t\mathcal{R}_{tb,\theta^{*}_m}(\overline{x}_{n}) \Big|^2\right)^{1/2}.
\end{cases}
\end{eqnarray*}
\item The spatial boundary condition training errors
    $\mathcal{E}_{T,sb}:=\mathcal{E}_{T,sb,0}+\mathcal{E}_{T,sb,1}+\mathcal{E}_{T,sb,2}+\mathcal{E}_{T,sb,3}$, where
\begin{eqnarray*}\label{training3}
\begin{cases}
\mathcal{E}_{T,sb,0}
:=\left(\sum_{n=1}^{N_{sb}}\omega_n^{sb,0} \Big|\mathcal{R}_{sb,\theta^{*}_m}(\widehat{x}_{n}, \widehat{t}_{n}) \Big|^2\right)^{1/2},\\
\mathcal{E}_{T,sb,1}
:=\left(\sum_{n=1}^{N_{sb}}\omega_n^{sb,1} \Big|\frac{\partial \mathcal{R}_{sb,\theta^*_m}}{\n}(\widehat{x}_{n}, \widehat{t}_{n})\Big|^2\right)^{1/2},\\
\mathcal{E}_{T,sb,2}
:=\left(\sum_{n=1}^{N_{sb}}\omega_n^{sb,2} \Big| \frac{\partial\mathcal{R}_{sb,\theta^{*}_m}(\widehat{x}_{n}, \widehat{t}_{n}) 
 }{\partial t} \Big|^2\right)^{1/2},\\
\mathcal{E}_{T,sb,3}
:=\left(\sum_{n=1}^{N_{sb}}\omega_n^{sb,3}\Big|\frac{ \partial^2 \mathcal{R}_{sb,\theta^*_m}}{\partial t\n}(\widehat{x}_{n}, \widehat{t}_{n})\Big|^2\right)^{1/2}.\\
\end{cases}
\end{eqnarray*}
\end{itemize}
We can compute these errors from the loss function \eqref{Loss1-1} using automatic differentiation in case of derivative terms.

The above derivations have established the following generalization error estimates for the inverse problem.

\begin{theorem}\label{theorem-error}
Under the assumption of Theorem \ref{theorem-uniq}, there exists a unique solution to the inverse problem \eqref{IP}.
Moreover, for the approximate solution $u_e^*$ of the forward problem with $\theta_e^*$ being a global minimizer of the loss function  $J_1^N(\theta_e)$, and $(\mu_f^*,u_m^*)$ of the inverse problem with $(\theta_f^*,\theta_m^*)$ being a global minimizer of the loss function  $J_2^N(\theta_f,\theta_m)$, we have the following generalization error estimates
\begin{equation}\label{IPue}
\begin{aligned}
\mathcal{E}_{G,u_e}
&\leq C\Big(\mathcal{\bar E}_{T,int}
+ \mathcal{\bar E}_{T,sb}
+ \mathcal{\bar E}_{T,tb}
+\bar C_{q}^{\frac{1}{2}} \bar N^{\frac{-\bar \alpha}{2}}
\Big),   \\
\mathcal{E}_{G,\vec p}
&\leq C\Big(\mathcal{\bar E}_{T,int}
+ \mathcal{\bar E}_{T,sb}
+ \mathcal{\bar E}_{T,tb} + \mathcal{E}_{T,d}+ \mathcal{E}_{T,int}
+ \mathcal{E}_{T,sb}
+ \mathcal{E}_{T,tb}
+C_{q}^{\frac{1}{2}} N^{\frac{-\alpha}{2}}
+\delta\Big),
\end{aligned}
\end{equation}
%\begin{eqnarray}\label{IPu}
%\mathcal{E}_{G,u}\leq
%\overline{C}\left( \mathcal{E}_{T,d}+ \mathcal{E}_{T,int}
%+ \mathcal{E}_{T,sb}
%+ \mathcal{E}_{T,tb}
%+C_{q}^{\frac{1}{2}} N^{\frac{-\alpha}{2}}
%+\delta\right),
%\end{eqnarray}
where the constants $C$ depend only on $\Omega$ and $T$. Furthermore, these constants are selected as
\begin{eqnarray*}
&&\bar C_{q}^{\frac{1}{2}} \bar N^{\frac{-\bar \alpha}{2}} = \max\Big\{ \bar C_{q}^{\frac{1}{2}} N_{int}^{\frac{- \bar \alpha_{int}}{2}},  \bar C_{qt}^{\frac{1}{2}}  N_{tb}^{\frac{- \bar \alpha_{tb}}{2}},  \bar C_{qs}^{\frac{1}{2}} N_{sb}^{\frac{- \bar \alpha_{sb}}{2}} \Big\},\\
&&C_{q}^{\frac{1}{2}} N^{\frac{-\alpha}{2}} = \max_{k=0,1,\ j=0,1,2,3}\Big\{\bar C_{q}^{\frac{1}{2}} \bar N^{\frac{-\bar \alpha}{2}}, C_{qd}^{\frac{1}{2}} N_{d}^{\frac{-\alpha_{d}}{2}}, C_{q}^{\frac{1}{2}} N_{int}^{\frac{-\alpha_{int}}{2}}, C_{qt,k}^{\frac{1}{2}} N_{tb}^{\frac{-\alpha_{tb,k}}{2}}, C_{qs,j}^{\frac{1}{2}} N_{sb}^{\frac{-\alpha_{sb,j}}{2}} \Big\},
\end{eqnarray*}
%\begin{align*}
%N&=\min \left\{N_{d}, N_{int},N_{sb},N_{tb}\right\},\\
%\alpha&=\min \left\{\alpha_{d,0},\alpha_{d,1}, \alpha_{int,0},\alpha_{int,1},\alpha_{sb,0},\alpha_{sb,1},\alpha_{sb,2},\alpha_{tb,0},\alpha_{tb,1},\alpha_{tb,2}\right\},\\
%\end{align*}
%in \eqref{Loss-error}, and
%$$C_{q}=\max\left\{C_{qd,0},C_{qd,1}, C_{q,0},C_{q,1},
%C_{qs,0},C_{qs,1},C_{qs,2},C_{qt,0},C_{qt,1}, C_{qt,2} \right\},$$
with the constant
\begin{align*}
\bar C_{q}=\bar C_{q}\Big( \Big\|\mathcal{R}_{int,\theta_e^*}\Big\|_{C(\Omega_T)}\Big),
\ \bar C_{qs}=\bar C_{qs}\Big(\Big\|\mathcal{R}_{sb,\theta_e^*}\Big\|_{C(\partial\Omega_T)}\Big),
\ \bar C_{qt}=\bar C_{qt}\Big(\Big\|\mathcal{R}_{tb,\theta_e^*}\Big\|_{C(\Omega)}\Big),
\end{align*}
and
\begin{align*}
&C_{qd}=C_{qd}\Big(\Big\|q^*\mathcal{R}_{d,\theta_m^*}\Big\|_{C(\Omega)}\Big),
&& C_{q}=C_{q}\Big(\Big\|\mathcal{R}_{int,\theta_e^*,\theta_f^*,\theta_m^*}\Big\|_{C(\Omega_T)}\Big),\\
&C_{qs,0}=C_{qs,0}\Big(\Big\|\mathcal{R}_{sb,\theta_m^*}\Big\|_{C(\partial\Omega_T)}\Big),
&& C_{qs,1}=C_{qs,1}\Big(\Big\|  \frac{\partial \mathcal{R}_{sb,\theta_m^*}}{\n}\Big\|_{C(\partial\Omega_T)}\Big),\\ &C_{qs,2}=C_{qs,2}\Big(\Big\| 
 \frac{\partial\mathcal{R}_{sb,\theta_m^*}}{\partial t}\Big\|_{C(\partial\Omega_T)}\Big),
&& C_{qs,3}=C_{qs,3}\Big(\Big\|\frac{\partial^2\mathcal{R}_{sb,\theta_m^*}}{\partial t\n}\Big\|_{C(\partial\Omega_T)}\Big),\\
&C_{qt,0}=C_{qt,0}\Big(\Big\|\mathcal{R}_{tb,\theta_m^*}\Big \|_{C(\Omega)}\Big),&& C_{qt,1}=C_{qt,1}\Big(\Big\| \frac{\partial\mathcal{R}_{tb,\theta_m^*}}{\partial t}\Big\|_{C(\Omega)}\Big),
\end{align*}
where we denote $\Omega_T:=\Omega\times(0,T)$ and $\partial\Omega_T:=\partial\Omega\times(0,T)$.
\end{theorem}

\begin{proof}
First, we will derive the estimates on the first formulation of \eqref{gener1}.
Introduce $\hat u_e:=u_e^*-u_{e,ex}$, where $u_e^*=u_{e,\theta_e^*}$ and $\hat  u_e(x,t)$ satisfies
\begin{equation*}
\begin{cases}
\begin{aligned}
\left(c^{-1}\partial_t +\A \right) \hat  u_e(x,t)&=\mathcal{R}_{{int},\theta_e^*},  &&(x,t)\in\Omega\times (0,T),\\
\hat u_e(x,0)&=\mathcal{R}_{tb,\theta_e^*}(x), &&x\in\Omega,\\
\mathcal B \hat  u_e(x,t)&=\mathcal{R}_{sb,\theta_e^*}(x, t), && (x,t)\in\partial\Omega\times(0,T).
\end{aligned}
\end{cases}
\end{equation*}
Using the regularity estimates in \cite{Evans}, there holds
\begin{equation}\label{regula-ue}
\begin{aligned}
\| \hat u_e\|_{L^\infty(0,T;L^{2}(\Omega))}
=& \|u_e^*-u_{e,ex}\|_{L^\infty(0,T;L^{2}(\Omega))} \\
\leq& C \| \mathcal{R}_{{int},\theta_e^*}\|_{L^2(0,T;L^{2}(\Omega))}
+ \| \mathcal{R}_{tb,\theta_e^*}\|_{L^2(\Omega)}
+\|\mathcal{R}_{sb,\theta_e^*}\|_{L^2(0,T;H^{-3/2}(\partial\Omega))}.
\end{aligned}
\end{equation}

Then, we will derive the estimates on the second formulation of \eqref{gener1}.
Introducing $\hat u_m:=u_m^*-u_{m,ex}$ with $$u_m^*=u_{m,\theta_m^*}=u_{m,\theta_m^*}[\mu_{f,\theta_f^*},u_{e,\theta_e^*}]
=u^*_{m}[\mu^*_{f},u^*_{e}],$$ we have the equation
\begin{equation*}\label{sensitivitypro}
\begin{cases}
\begin{aligned}
\left(c^{-1}\partial_t +\A \right) \hat  u_m(x,t)&=\mu^*_{f}u^*_{e}-\mu_{f,ex}
 u_{e,ex}+\mathcal{R}_{{int},\theta_m^*,\theta_f^*,\theta_e^*},  &&(x,t)\in\Omega\times (0,T),\\
\hat  u_m(x,0)&=\mathcal{R}_{tb,\theta_m^*}(x), &&x\in\Omega,\\
\mathcal B \hat  u_m(x,t)&=\mathcal{R}_{sb,\theta_m^*}(x, t), && (x,t)\in\partial\Omega\times(0,T),
\end{aligned}
\end{cases}
\end{equation*}
and the observation data
\begin{eqnarray*}\label{sensitivityObs}
\frac{\partial \hat u_m}{\n}=\frac{\partial u_{m,\theta_m^*}}{\n}-\frac{\partial u_{m,ex}}{\n}=\mathcal{R}_{d, \theta_m^*}(x,t)+\varphi^\delta-\varphi, \quad  (x,t) \in \Gamma\times(0,T].
\end{eqnarray*}
We make the decomposition $\hat u_m:=\hat u_{m,1}+\hat u_{m,2}$, where $\hat u_{m,1}$, $\hat u_{m,2}$ satisfy
\begin{equation*}\label{sensitivitypro1}
\begin{cases}
\begin{aligned}
\left(c^{-1}\partial_t +\A \right) \hat u_{m,1}(x,t)&=\mu^*_{f}u^*_{e}-\mu_{f,ex}
 u_{e,ex},  &&(x,t)\in\Omega\times (0,T),\\
\hat u_{m,1}(x,0)&=0, &&x\in\Omega,\\
\mathcal B \hat u_{m,1}(x,t)&=0, && (x,t)\in\partial\Omega\times(0,T),
\end{aligned}
\end{cases}
\end{equation*}
with
\begin{eqnarray*}\label{sensitivityObs1}
\frac{\partial \hat u_{m,1}}{\n}=\frac{\partial \hat u_m}{\n}-\frac{\partial \hat u_{m,2}}{\n}=\mathcal{R}_{d, \theta_m^*}(x,t)+(\varphi^\delta-\varphi)-\frac{\partial \hat u_{m,2}}{\n}, \quad  (x,t) \in \Gamma\times(0,T],
\end{eqnarray*}
and
\begin{equation*}\label{sensitivitypro2}
\begin{cases}
\begin{aligned}
\left(c^{-1}\partial_t +\A \right) \hat u_{m,2}(x,t)&=\mathcal{R}_{{int},\theta_m^*,\theta_f^*,\theta_e^*},  &&(x,t)\in\Omega\times (0,T),\\
\hat u_{m,2}(x,0)&=\mathcal{R}_{tb,\theta_m^*}(x), &&x\in\Omega,\\
\mathcal B \hat u_{m,2}(x,t)&=\mathcal{R}_{sb,\theta_m^*}(x, t), && (x,t)\in\partial\Omega\times(0,T),
\end{aligned}
\end{cases}
\end{equation*}
respectively.

Next we split $\hat u_{m,1}=\hat u_{m,11}+\hat u_{m,12}$, where $\hat u_{m,11}$, $\hat u_{m,12}$ satisfy
\begin{equation}\label{sensitivitypro1-1}
\begin{cases}
\begin{aligned}
\left(c^{-1}\partial_t +\A \right) \hat u_{m,11}(x,t)&=(\mu^*_{f}-\mu_{f,ex}
 )u_{e,ex},  &&(x,t)\in\Omega\times (0,T),\\
\hat u_{m,11}(x,0)&=0, &&x\in\Omega,\\
\mathcal B \hat u_{m,11}(x,t)&=0, && (x,t)\in\partial\Omega\times(0,T),
\end{aligned}
\end{cases}
\end{equation}
with
\begin{equation}\label{sensitivityObs1-1}
\begin{aligned}
\frac{\partial \hat u_{m,11}}{\n}
&=\frac{\partial \hat u_{m,1}}{\n}-\frac{\partial \hat u_{m,12}}{\n} \\
&=\mathcal{R}_{d, \theta_m^*}(x,t)+(\varphi^\delta-\varphi)-\frac{\partial \hat u_{m,2}}{\n}-\frac{\partial \hat u_{m,12}}{\n},\; (x,t) \in \Gamma\times(0,T],
\end{aligned}
\end{equation}
and
\begin{equation*}\label{sensitivitypro2-1}
\begin{cases}
\begin{aligned}
\left(c^{-1}\partial_t +\A \right) \hat u_{m,12}(x,t)&=\mu^*_{f}(u^*_{e}-u_{e,ex}),  &&(x,t)\in\Omega\times (0,T),\\
\hat u_{m,12}(x,0)&=0, &&x\in\Omega,\\
\mathcal B \hat u_{m,12}(x,t)&=0, && (x,t)\in\partial\Omega\times(0,T).
\end{aligned}
\end{cases}
\end{equation*}
For the inverse problem \eqref{sensitivitypro1-1}-\eqref{sensitivityObs1-1},
according to conditional stability result, i.e, Theorem \ref{stability-inv}, we get
\begin{align}\label{stability-use}
\| {\vec p_*-\vec p_{ex}}\|_{\mathcal{L}}
\leq& C\Big\|\frac{\partial \hat u_{m,11}}{\n}\Big\|_{L^2(\Gamma\times(0,T))} \nonumber \\
=& C \Big\| \mathcal{R}_{d, \theta_m^*}(x,t)+(\varphi^\delta-\varphi)-\frac{\partial \hat u_{m,2}}{\n}-\frac{\partial \hat u_{m,12}}{\n}\Big\|_{L^2(\Gamma\times(0,T))} \nonumber\\
\leq& \| \mathcal{R}_{d, \theta_m^*}\|_{L^2(\Gamma\times(0,T))}
+ \| \varphi^\delta-\varphi\|_{L^2(\Gamma\times(0,T))}
+ \Big\| \frac{\partial \hat u_{m,2}}{\n}\Big\|_{L^2(\Gamma\times(0,T))}
+ \Big\| \frac{\partial \hat u_{m,12}}{\n} \Big\|_{L^2(\Gamma\times(0,T))} \nonumber\\
\leq& \| \mathcal{R}_{d, \theta_m^*}\|_{L^2(\Gamma\times(0,T))}
+ \Big\| \frac{\partial \hat u_{m,2}}{\n}\Big\|_{L^2(\partial\Omega\times(0,T))}
+ \Big\| \frac{\partial \hat u_{m,12}}{\n} \Big\|_{L^2(\partial\Omega\times(0,T))} + \delta.
%&\leq& \Big\| \mathcal{R}_{d, \theta_m^*}\Big\|_{L^2(\partial\Omega\times(0,T))}
%+ \Big\| \hat u_2 \Big\|_{L^2(0,T;H^{3/2}(\Omega))}
%+ \Big\| \hat u_{12} \Big\|_{L^2(0,T;H^{3/2}(\Omega))}+ \delta \nonumber
\end{align}
Using trace theorem and Sobolev embedding theorem, we further get
\begin{align}\label{stability-use-1}
\| {\vec p_*-\vec p_{ex}}\|_{\mathcal{L}}
\leq&\| \mathcal{R}_{d, \theta_m^*}\|_{L^2(\Gamma\times(0,T))}
+ \| \hat u_{m,2} \|_{L^2(0,T;H^{3/2}(\Omega))}
+ \| \hat u_{m,12} \|_{L^2(0,T;H^{3/2}(\Omega))}+ \delta \nonumber\\
\leq& \| \mathcal{R}_{d, \theta_m^*}\|_{L^2(\Gamma\times(0,T))}
+ \| \hat u_{m,2} \|_{L^2(0,T;H^{2}(\Omega))}
+ \| \hat u_{m,12} \|_{L^2(0,T;H^{2}(\Omega))}+ \delta.
\end{align}
According to the estimates in \cite{Evans}, the second term of the inequality \eqref{stability-use-1} satisfies
\begin{align}\label{stability-use-2}
\| \hat u_{m,2} \|_{L^2(0,T;H^{2}(\Omega))}
\leq& C\Big( \| \mathcal{R}_{{int},\theta_m^*,\theta_f^*,\theta_e^*}\|_{L^2(0,T;L^{2}(\Omega))} +\| \mathcal{R}_{tb,\theta_m^*}\|_{H_0^1(\Omega)}
 \nonumber \\
&\ \quad +\|\mathcal{R}_{sb,\theta_m^*}\|_{L^2(0,T;{H^{1/2}(\partial\Omega)})}
+ \|\partial_t \mathcal{R}_{sb,\theta_m^*}\|_{L^2(0,T;{H^{1/2}(\partial\Omega)})}\Big )\nonumber\\
\leq& C \Big( \| \mathcal{R}_{{int},\theta_m^*,\theta_f^*,\theta_e^*}\|_{L^2(0,T;L^{2}(\Omega))}
+\| \mathcal{R}_{tb,\theta_m^*}\|_{H_0^1(\Omega)}
+ \|\mathcal{R}_{sb,\theta_m^*}\|_{H^1(0,T;{H^{1}(\partial\Omega)})}\Big).
\end{align}
Using \eqref{regula-ue}, the third term in \eqref{stability-use-1} satisfies
\begin{align}\label{stability-use-3}
\| \hat u_{m,12} \|_{L^2(0,T;H^{2}(\Omega))}
\leq& C \| \mu^*_{f}(u^*_{e}-u_{e,ex})\|_{L^2(0,T;L^{2}(\Omega))}
\nonumber \\
\leq& C \| (u^*_{e}-u_{e,ex})\|_{L^\infty(0,T;L^{2}(\Omega))}
\nonumber \\
\leq& C \Big( \| \mathcal{R}_{{int},\theta_e^*}\|_{L^2(0,T;L^{2}(\Omega))}
+\| \mathcal{R}_{tb,\theta_e^*}\|_{L^2(\Omega)}
+ \| \mathcal{R}_{sb,\theta_e^*}\|_{L^2(0,T;H^{-3/2}(\partial\Omega))}\Big)
\nonumber \\
\leq& C \Big( \| \mathcal{R}_{{int},\theta_e^*}\|_{L^2(0,T;L^{2}(\Omega))}
+ \| \mathcal{R}_{tb,\theta_e^*}\|_{L^2(\Omega)}
+ \| \mathcal{R}_{sb,\theta_e^*}\|_{L^2(0,T;L^{2}(\partial\Omega))}\Big).
\end{align}
Combining \eqref{stability-use}-\eqref{stability-use-3} together and using the Sobolev embedding theorem, we get the generalization error estimate for $\vec {p}$
\begin{align}\label{stability-use-final-1}
\| {\vec p_*-\vec p_{ex}}\|_{\mathcal{L}}\leq& C \Big(
\| \mathcal{R}_{{int},\theta_e^*}\|_{L^2(0,T;L^{2}(\Omega))}
+ \| \mathcal{R}_{tb,\theta_e^*}\|_{L^2(\Omega)}
+ \|\mathcal{R}_{sb,\theta_e^*}\|_{L^2(0,T;L^2(\partial\Omega))}+\delta
 \nonumber \\
&\quad+\| \mathcal{R}_{d, \theta_m^*}\|_{L^2(\Gamma\times(0,T))}
+\|\mathcal{R}_{{int},\theta_m^*,\theta_f^*,\theta_e^*}\|_{L^2(0,T;L^{2}(\Omega))}
+\| \mathcal{R}_{tb,\theta_m^*}\|_{H_0^1(\Omega)} \nonumber \\
&\quad+\|\mathcal{R}_{sb,\theta_m^*}\|_{H^1(0,T;{H^{1}(\partial\Omega)})}\Big).
%\right.+ \nonumber \\&&\left.\quad\;
% \|\mathcal{R}_{sb,\theta_m^*}\|_{L^2(0,T;{H^{1/2}(\partial\Omega)})}
%+ \|\partial_t \mathcal{R}_{sb,\theta_m^*}\|_{L^2(0,T;{H^{1/2}(\partial\Omega)})}
\end{align}

The formulations in \eqref{IPue} are generated by applying \eqref{regula-ue} and \eqref{stability-use-final-1}, with the quadrature rules \eqref{quadrature} respectively. The proof is complete.
\end{proof}

The above estimates include the difference between the underlying solution and the approximation solution produced by data-driven method of the inverse problem, which represent the approximation of the problem in a finite dimensional space, and essentially reflect the stability due to both the model itself and the reconstruction scheme. Therefore, based on the result of Theorem \ref{theorem-error}, we can conduct the reconstruction algorithms both for excitation process to solve the direct problem and emission process to solve the inverse problem. 
 
\begin{remark}
In this work, the attempt to evaluate the generalization error estimate for unknown $u_m$ employing the generalization error \eqref{stability-use-final-1} is frustrated. This is because the norm equality between $\|\vec p\|_{\mathcal{L}}$ and $\|L \vec p \|_{L^2(0,T;L^{2}(\Omega))}$ is not proven. In the future, we will explore the suitable norm on $\vec p$ and give the generalization error estimate for $u_m$.
\end{remark}

\section{Numerical inversions.}\label{sec_num}

For actual implementation, we firstly parameterize $u_e$ by deep neural networks $u_{e,\theta_e}$  with network parameters $\theta_e$. Minimizing the loss function \eqref{loss-ue} to search for the optimal parameters $\theta^*_e$. Then, for the obtained $u_e^*:=u_{e,\theta_e^{*}}$, we will consider to recover the solution $(\mu_f, u_m)$ for inverse problems.
We also parameterize $\mu_f$ and $u_m$ by deep neural networks $\mu_{f,\theta_f}$ and $u_{m,\theta_m}$ with network parameters $\theta_f$ and $\theta_m$, respectively. Minimizing the loss function \eqref{loss-um} to search for the optimal parameters $\theta_f^*$ and $\theta_m^*$, the corresponding $\mu_f^* = \mu_{f,\theta_f^*}$ and $u_m^* = u_{m,\theta_m^*}$ are the approximate solution of the inverse problem. We construct the next Algorithms \ref{alg0} and \ref{alg1} to solve the direct problem \eqref{excitation-Ue} and the inverse problem \eqref{IP} respectively. 

\begin{algorithm}[H]
\caption{
  Data-driven solution of the direct problem \eqref{excitation-Ue}.}
\label{alg0}
\begin{algorithmic}
\REQUIRE
 Boundary input $g(x,t)$ for the excitation process \eqref{excitation-Ue}.

\STATE {\bf Initialize}
Network architectures $\ u_{e,\theta_e}$ and parameters $\theta_e$.
\FOR {$j=1,\cdots,K_1$}
 \STATE
Sample
$\ \mathcal{S}_{int},\ \mathcal{S}_{sb},\ \mathcal{S}_{tb}$.
\STATE
$\theta_e  \leftarrow \operatorname{Adam}\left(-\nabla_{\theta_e} J_1^N(\theta_e), \tau_{\theta_e}\right),$

\ENDFOR
\ENSURE
$ u_{e,\theta_e^*}$.
\end{algorithmic}
\end{algorithm}

\begin{algorithm}[H]
\caption{
  Data-driven solution of the inverse problem \eqref{IP}.}
\label{alg1}
\begin{algorithmic}
\REQUIRE
 The obtained $u_{e,\theta_e^*}$ and noisy measurement data $\varphi^\delta$ for the inverse problem.

\STATE {\bf Initialize}
Network architectures $\left(\mu_{f,\theta_f}, u_{m,\theta_m}\right)$ and parameters $(\theta_f,\theta_m)$.
\FOR {$j=1,\cdots,K_2$}
 \STATE
Sample
$\mathcal{S}_{d},\ \mathcal{S}_{int},\ \mathcal{S}_{sb},\ \mathcal{S}_{tb}$.
\STATE
$\theta_f \leftarrow \operatorname{Adam}\left(-\nabla_{\theta_f}  J_2^N(\theta_f, \theta_m), \tau_{\theta_f}, \lambda\right)$,\\
$\theta_m  \leftarrow \operatorname{Adam}\left(-\nabla_{\theta_m} J_2^N(\theta_f, \theta_m), \tau_{\theta_m}, \lambda\right),$

\ENDFOR
\ENSURE
$\left(\mu_{f,\theta_f^*}, u_{m,\theta_m^*}\right)$.
\end{algorithmic}
\end{algorithm}

The above minimization problems are to search the minimizer of the possibly non-convex function
 $J_1^N(\theta_e)$ and $J_2^N(\theta_f,\theta_m)$ over $\Theta\subset \mathbb{R}^\mathcal{M}$ for possibly very large $\mathcal{M}$.
 The hyper-parameters $( \tau_{\theta_e},\tau_{\theta_f},\tau_{\theta_m})$ are learning rates and $\lambda$
 are balance hyper-parameters between PDE and measurement data residuals. The optimizer is Adam (Adaptive Moment Estimation), which is an optimization algorithm commonly used in deep learning for training neural networks.
The robust analysis for hyper-parameters $\lambda$ will be studied in the numerical implementation subsection.

\paragraph{Example 1:}
The boundary input of excitation process is set as
$$ \mathcal B u_e = -20tx(x-1), \;\;(x,y,t)\in  \partial\Omega\times(0,T),$$
where the diffusion domain is $\Omega=(0,1)^2$, the final time is $T=1$ and the exact source $\mu_f$ of emission system is given as
\begin{eqnarray*}\label{eg1}
\mu_f(x,y,t)= 5 + t + \cos(\pi x)\cos(\pi y),
\end{eqnarray*}
which is smooth on the whole time and space domain $\Omega\times(0,T)$. The exact measurement will be
\begin{eqnarray*}\label{SS2}
\varphi(x,y,t)=\frac{\partial u_m}{\n}\Big|_{\Gamma \times(0,T)},
\end{eqnarray*}
with $\Gamma\subset\partial\Omega$, and in our experiments the noisy data is set as
\begin{eqnarray}\label{random}
\varphi^\delta(x,y,t):=\varphi(x,y,t)+\delta\cdot (2\; \text{rand}(\text{shape}(\varphi(x,y,t)))-1),
\end{eqnarray}
where $\text{rand}(\text{shape}(\varphi))$ is a random variable generated by uniform distribution in $[0,1]$.

For the implementation details, we use a fully connected neural network for $u_{e,\theta_e}$ with 3 hidden layers, each layer with a width of 20. We take
$\text{N}=N_{int}+N_{sb}+N_{tb}=256+256\times4+256=1536$ as the number of collocation points, which are randomly sampled in three different domains, i.e., interior spatio-temporal domain, spatial and temporal boundary domain. The collocation points are selected with respect to uniform distributions. The activation function is $tanh$. The number of training epochs is set to be $5\times10^4$, and the initial learning rates start with $0.001$ and shrink $10$ times every $2\times10^4$ iteration.
For the inversion of emission process, we also use the fully connected neural network both for $\mu_{f,\theta_f}$ and $u_{m,\theta_m}$ with 3 hidden layers, each layer with a width of 20. The number of training points is $\text{N}=N_{int}+N_{sb}+N_{tb}+N_{d}=500+500\times4+500+500=3500$, which are randomly sampled in four different domains, i.e., interior spatio-temporal domain, spatial and temporal boundary domain, and measurement domain. The collocation points are selected with respect to uniform distributions. The activation functions for $\mu_{f,\theta_f}$ and $u_{m,\theta_m}$ are both $tanh$, and the hyper-parameter is $\lambda=100$. The number of training epochs is set to be $2\times10^4$, and the initial learning rates start with $0.001$ and shrink $10$ times every $2000$ iteration.
The test sets are chosen by a uniform mesh
\begin{eqnarray}\label{testset}
\mathcal{T}:=\{(t_k,x_i,y_j):  k,i,j=0,1,\cdots,49\}\subset \Omega_T.
\end{eqnarray}

\begin{figure}[h!]
\centering
\includegraphics[width=1.1\textwidth,height=0.45\textheight,center]{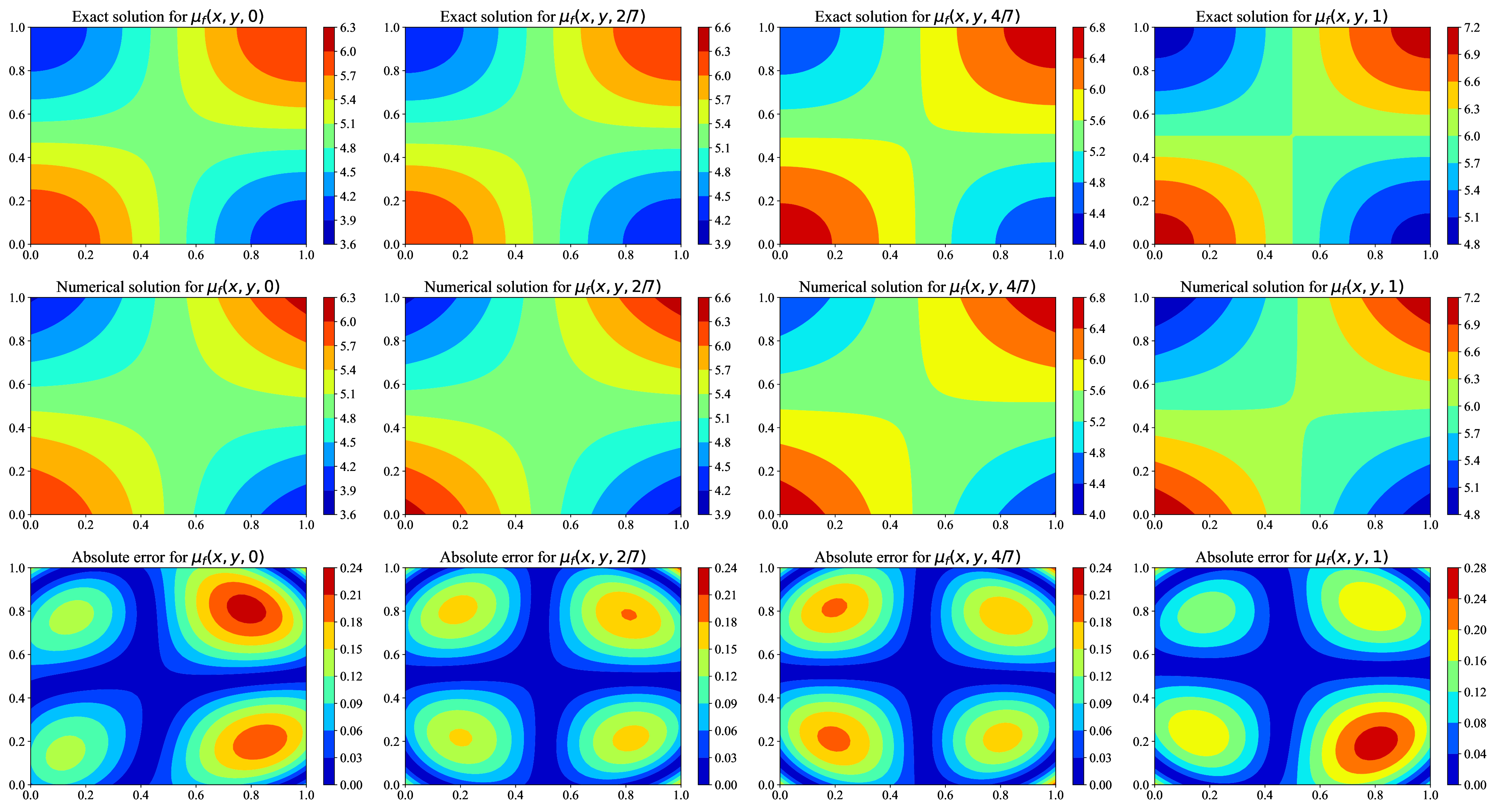}
\caption
{The exact (upper), reconstruction results (middle) and corresponding absolute pointwise errors (bottom)  for absorption coefficient $\mu_f$ at different times with noisy level $\delta=0.01$.}\label{PINN-q-exact-2D-rev1}
\end{figure}

\begin{figure}[h!]
\centering
\includegraphics[width=1.1\textwidth,height=0.3\textheight,center]{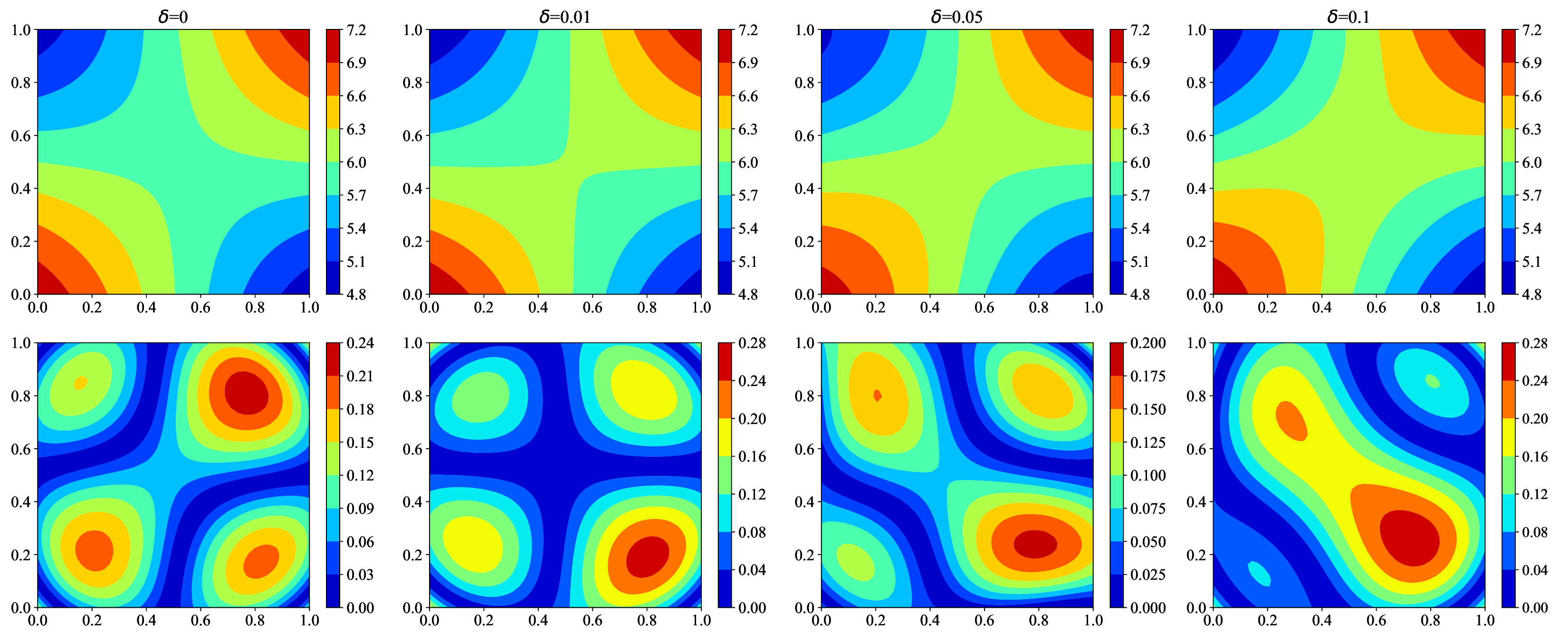}
\caption
{The reconstruction results (upper line) and corresponding absolute pointwise errors  (bottom line) for absorption coefficient $\mu_f(x,y,1)$ with different noisy levels $\delta=0,\; 0.01, \;0.05,\; 0.1$.}\label{PINN-2D-muf-rev1}
\end{figure}

Figure \ref{PINN-q-exact-2D-rev1} shows the exact (upper line), the reconstruction results (middle line) for absorption coefficient $\mu_f$ and the corresponding absolute pointwise error (bottom line) at different times $t=0,\; 2/7, \;4/7, \; 1$, respectively, with the noisy level $\delta=0.01$. We can find that the reconstruction results for  $\mu_f$ is well even with high noise of measurement data. 
Figure \ref{PINN-2D-muf-rev1} shows the reconstruction results (upper line) and the corresponding absolute pointwise error (bottom line) for $\mu_f$ at fixed time $t=1$ with different noisy level $\delta=0,\; 0.01, \; 0.05, \; 0.1$, respectively.
Meanwhile, the numerical solution  for excitation process $u_e$ (upper line) are shown in Figure \ref{PINN-2D-ue-um-rev1}, which also presents the reconstruction of emission solution $u_m$ at different times $t=2/7$, $t=3/7$, $t=5/7$ and $t=1$, respectively, with various noisy levels $\delta=0$ (second line), $\delta=0.01$ (third line) and $\delta=0.1$ (bottom line). We can see the evolution of the excitation solution for $u_e$ over time, and the reconstruction accuracy for $u_m$ deteriorates as the noise level increasing, but the performance is still satisfactory.

\begin{figure}[h!]
\centering
\includegraphics[width=1.1\textwidth,height=0.6\textheight,center]{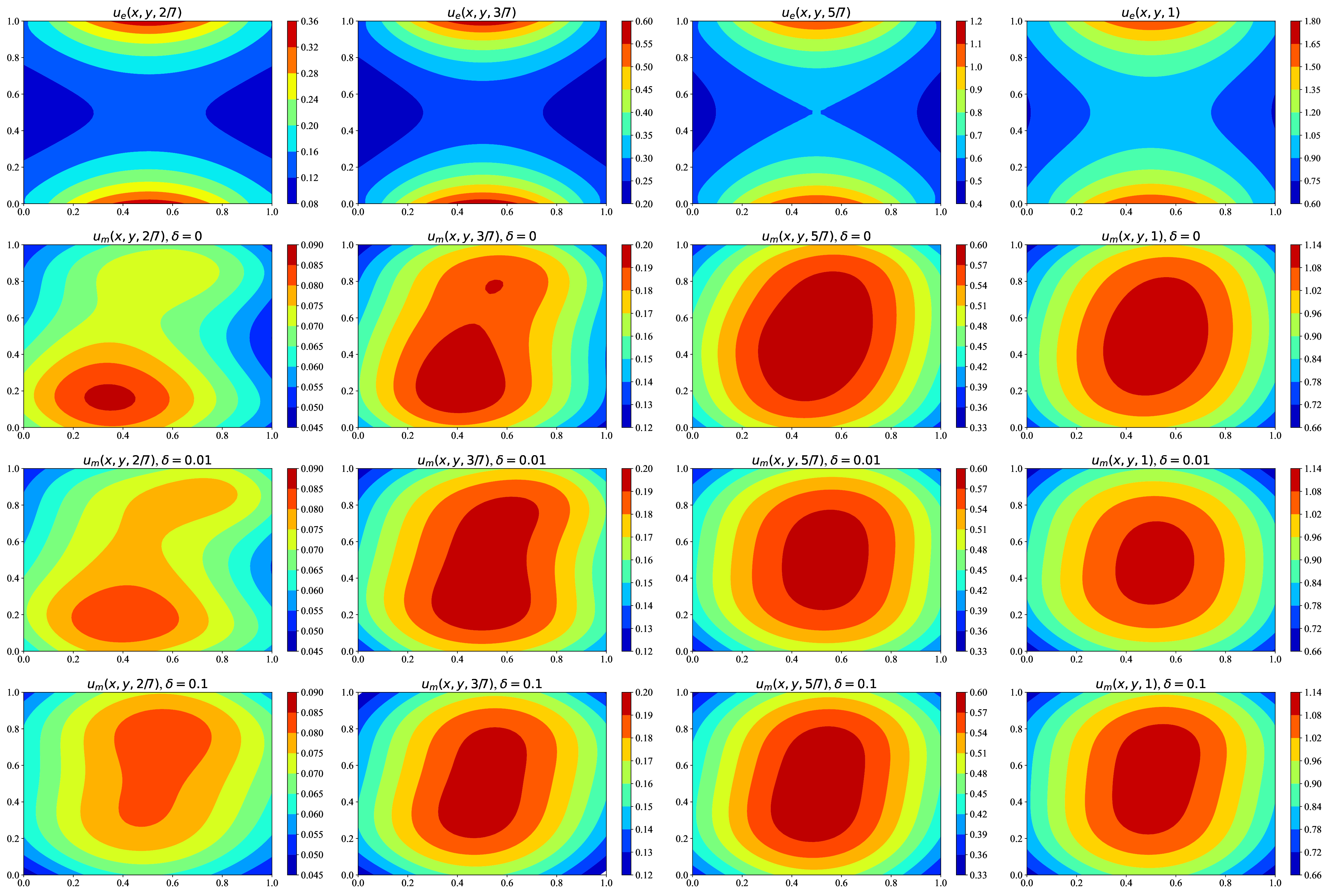}
\caption
{The numerical results for $u_e$ and $u_m$ at different times $t=2/7, \;3/7, \; 5/7,\; 1$ with various noisy levels $\delta=0,\; 0.01,\; 0.1$.}\label{PINN-2D-ue-um-rev1}
\end{figure}

\paragraph{Example 2:}
The boundary input of excitation process is set as
$$ \mathcal B u_e = -20tx(x-1), \;\;(x,y,t)\in  \partial\Omega\times[0,T],$$
where the diffusion domain is $\Omega=(0,1)^2$, the final time is $T=1$ and the exact source $\mu_f$ of emission system is given as
\begin{eqnarray*}\label{eg2}
\mu_f(x,y,t)= (t+1)f(r)
\end{eqnarray*}
with
\begin{equation*}
 \begin{aligned}
  f(r)&=
\begin{cases}
15\left(\cos r-\sqrt{3}/2\right)+2, \quad 0\leq r \leq \pi/6, \\
2, \quad \text { otherwise },
\end{cases}\\
r(x,y)&=\sqrt{(x-0.5)^2+(y-0.5)^2},
 \end{aligned}
\end{equation*}
which is continuous but not smooth in spatial domain. The noisy measurement data is generated by the same way in \eqref{random}.

For the implementation details of excitation process, the architectures of the neural networks, the training points sampling and the test sets are the same as Example 1 to get $u_{e,\theta_e}$.
For the inversion of emission process, we also use the fully connected neural network both for $\mu_{f,\theta_f}$ and $u_{m,\theta_m}$ with 3 hidden layers, each layer with a width of 20. The number of training points is $\text{N}=N_{int}+N_{sb}+N_{tb}+N_{d}=500+500\times4+500+500=3500$, which are randomly sampled in four different domains, i.e., interior spatio-temporal domain, spatial and temporal boundary domain, and measurement domain. The collocation points are selected with respect to uniform distributions. The activation functions for $\mu_{f,\theta_f}$ and $u_{m,\theta_m}$ are both $tanh$, and the hyper-parameter is $\lambda=100$. The number of training epochs is set to be $4\times10^4$, and the initial learning rates start with $0.002$ and shrink $10$ times every $2\times10^4$ iteration.
The test sets are chosen by a uniform mesh
\eqref{testset}.

Since the noisy level of the measurement data affects the reconstruction accuracy, in this simulation, we test the training performance for various noisy levels. Figure \ref{PINN_Training-2D-1} records the training process, i.e., the training loss, the relative error for the reconstruction of $\mu_f$ with respect to the iterations for different noise levels $\delta=0,\;0.1\%,\;1\%,\;10\%$ by the proposed scheme.
% After training, we test the reconstruction result on test sets $\mathcal{T}$.

\begin{figure}[h!]
\centering
\includegraphics[width=0.8\textwidth,height=0.2\textheight]{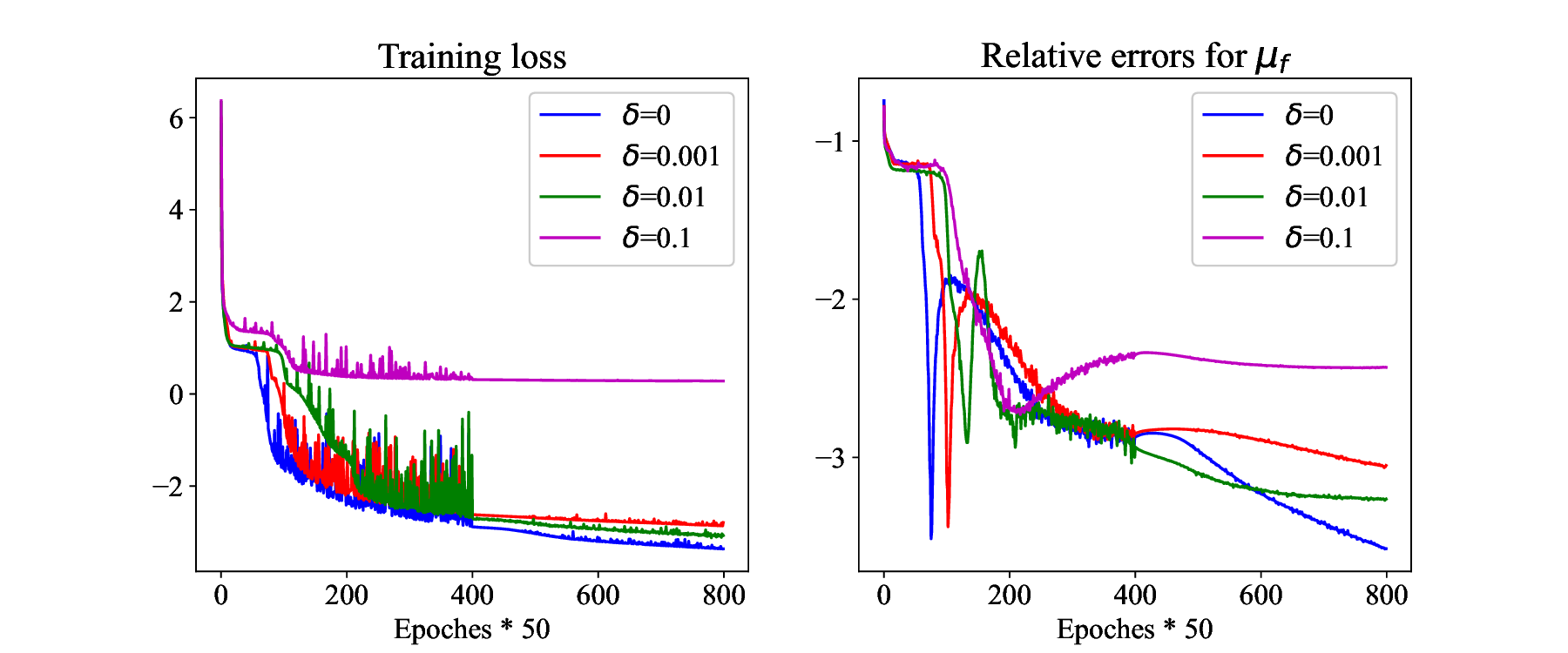}
\caption
{The training loss (left) and the relative error for $\mu_f$ (right) after logarithmic re-scaling.
%\QL{this figure exceeds margins}
}\label{PINN_Training-2D-1}
\end{figure}

The distribution of the absorption coefficient $\mu_f(x,t)$ also depends on the time $t$, Figure \ref{PINN-time-2D-1} shows the time-series relative error of the recovered $\mu_f$ with various noise levels. As shown in this figure, the training performance deteriorates as the noise level of the measurement data increasing.
\begin{figure}[h!]
\centering
\includegraphics[width=0.3\textwidth,height=0.17\textheight]{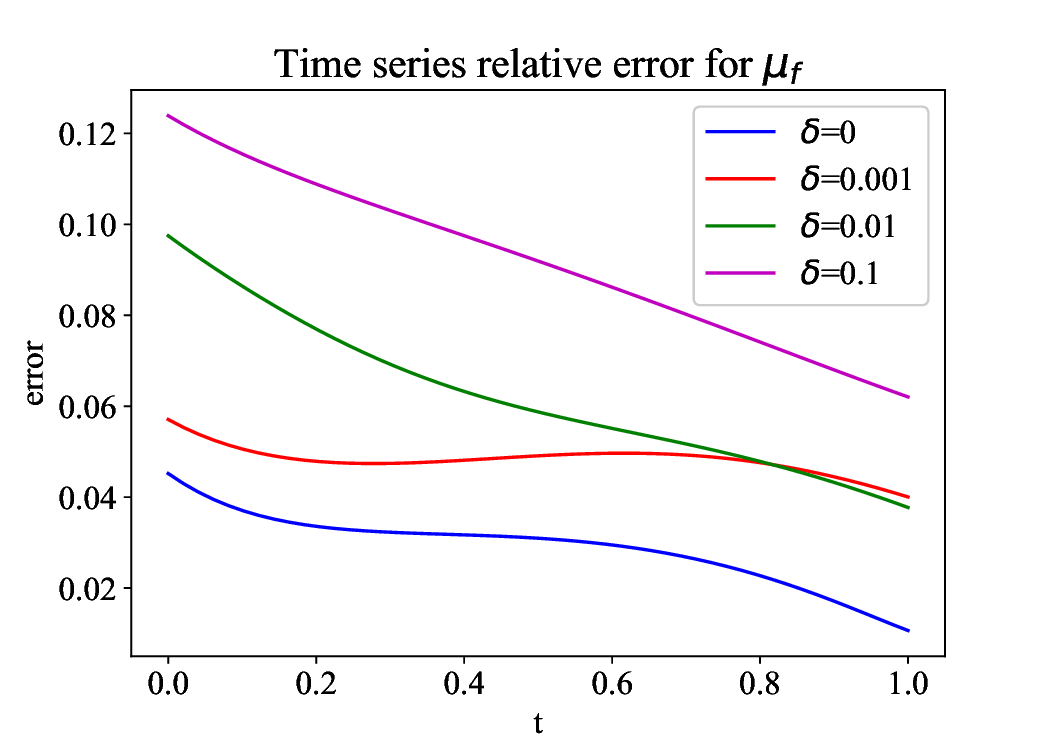}
\caption{The time series relative error (test) for $\mu_f$ after logarithmic re-scaling  for different noisy data.}\label{PINN-time-2D-1}
\end{figure}

 Figure \ref{PINN-q-exact-2D-1} shows the exact absorption coefficient at final time $t=1$. Figure \ref{PINN-q-2D-1} shows the reconstruction results for $\mu_f$ at final time $t=1$ by optimizing the proposed loss function (upper line) and the corresponding absolute pointwise error (bottom line) for various noisy levels $\delta=0, \; 0.1\%,\; 1\%,\; 10\%$. Meanwhile, the numerical solution $u_e$ at different times for excitation process are shown in Figure \ref{PINN-ue-exact-2D-1}. Figure \ref{PINN-um-2D-T} presents the reconstruction of emission solution $u_m$ at different times $t=2/7$, $t=3/7$ and $t=1$, respectively, with various noisy levels. We can see that the reconstruction accuracy deteriorates as the noise level of the measurement data increasing, but the performance is still satisfactory.

\begin{figure}[h!]
\centering
\includegraphics[width=0.38\textwidth,height=0.2\textheight]{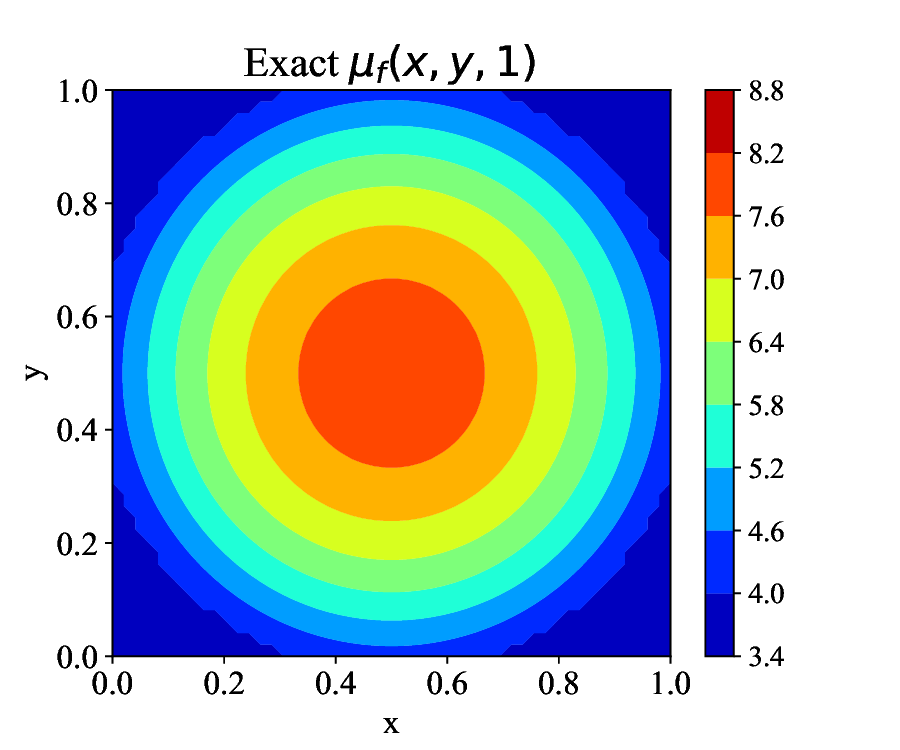}
\caption
{The exact absorption coefficient $\mu_f$ at final time $t=1$.}\label{PINN-q-exact-2D-1}
\end{figure}

\begin{figure}[h!]
\centering
\includegraphics[width=1.1\textwidth,height=0.3\textheight,center]{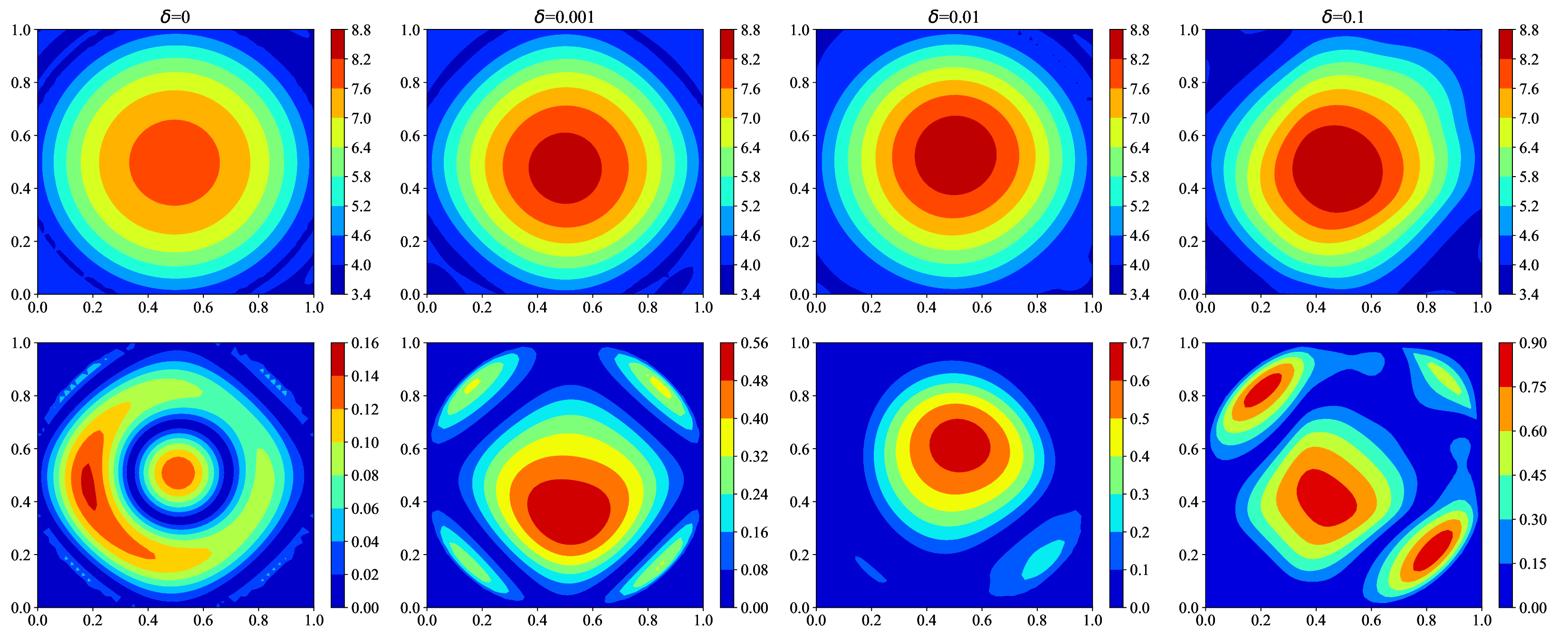}
\caption{The reconstruction of absorption coefficient $\mu_f$ (upper line) and corresponding absolute point-wise error $|\mu_f-\mu_f^*|$ (bottom line) for various noisy level measurement data at final time $t=1$.}\label{PINN-q-2D-1}
\end{figure}

%\begin{figure}[H]
%\centering
%\includegraphics[width=0.6\textwidth,height=0.15\textheight,center]{PINN_ue_um.eps}
%\caption{The reconstruction of the excitation solution (left) and emission solution (right) at final time $t=1$.}\label{PINN-ue-um}
%\end{figure}

\begin{figure}[h!]
\centering
\includegraphics[width=1\textwidth,height=0.18\textheight]{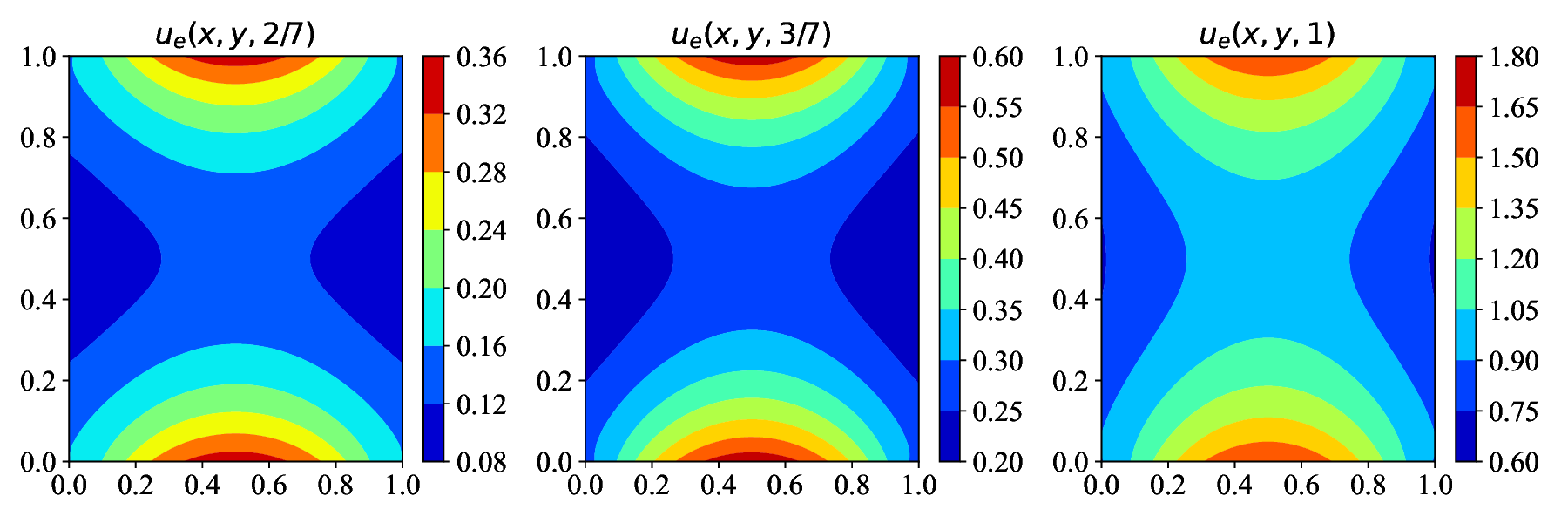}
\caption
{The numerical solution $u_e$ for excitation process at different times $t=2/7$ (left), $t=3/7$ (middle) and at final time $t=1$ (right), respectively.}\label{PINN-ue-exact-2D-1}
\end{figure}

\begin{figure}[h!]
\centering
\includegraphics[width=1.1\textwidth,height=0.47\textheight,center]{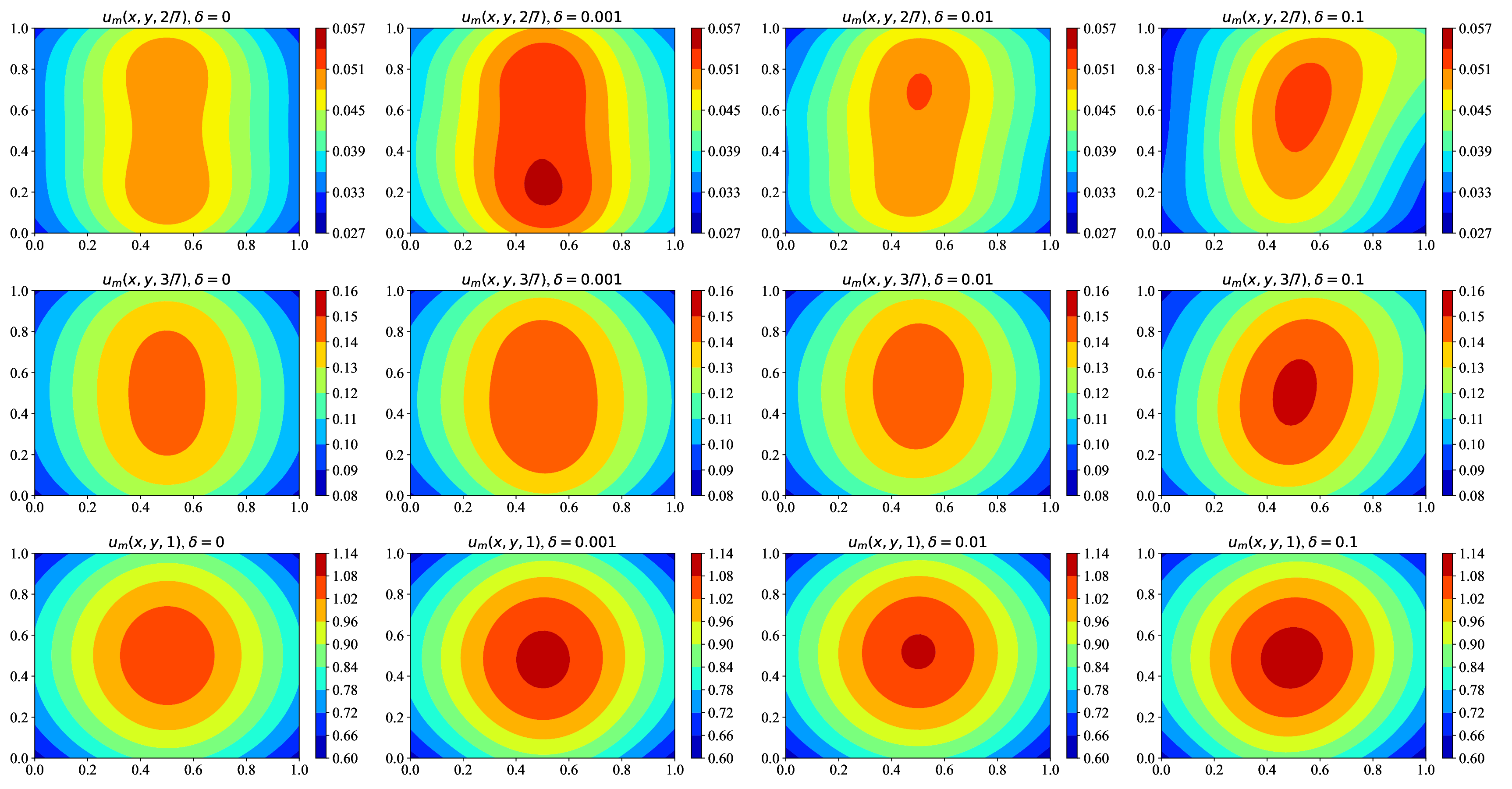}
\caption{The reconstruction of the emission solution $u_m$ at different times $t=2/7$ (upper line), $t=3/7$ (middle line) and at final time $t=1$ (bottom line), respectively, with various noisy levels $\delta=0, \; 0.1\%,\; 1\%,\; 10\%$ .}\label{PINN-um-2D-T}
\end{figure}

In order to evaluate the effectiveness of the proposed scheme in terms of hyper-parameters, some experiments are conducted. Specifically, we examine the impact of the balance hyper-parameter $\lambda$ in \eqref{loss-um}. For the fixed hidden layers ($\text{NL}=3$) and fixed neurons per-layer ($\text{NN}=20$), we compute the reconstruction errors (mean and standard deviation) for $\mu_f$ and $u_m$ with different values of $\lambda$, such as
$10000, 1000, 100, 10, 1, 0.1 $. The results of these experiments are presented in Table \ref{tab4-D-Ab1-0}, which indicate that the performance of the inverse problem is highly dependent on the balance hyper-parameter $\lambda$. Specifically, we find that the relative reconstruction errors are optimized when $\lambda$ is between $10$ to $1000$. Furthermore, we observe that the reconstruction errors increase significantly as $\lambda$ exceeds this optimal value. These results suggest that the selection of the balance hyper-parameter is critical to achieving good performance in this inverse problem.

%Moreover, we refer to some recent works which considered the choice of hyper-parameter $\lambda$. For example, Xiang et al \cite{XiangZixue} proposed a self-adaptive loss balanced method which automatically assigns the weights of losses by updating adaptive weights in each epoch. Ruia et al \cite{RuiEn-Ze} proposed a dynamic prioritization self-adaptive loss balance strategy, which adaptively reconciles the loss terms of distinct scales to facilitate convergence in PINN training.
%Recently, Afkham et al \cite{afkham2021learning} proposed a new approach that uses deep neural networks (DNN) to obtain regularization parameters for solving inverse problems, in which a network is trained to approximate the mapping from observation data to regularization parameters.

\begin{table}[h!]
\centering
\caption{The relative errors for $\mu_f$ using different $\lambda$ ($\delta=0.01$, $\text{NL}=3$, $\text{NN}=20$).} \label{tab4-D-Ab1-0}
\begin{tabular}{cccccccccc}
\specialrule{0.05em}{13pt}{3pt}
   Error &$\lambda=10000$  &$\lambda=1000$  &$\lambda=100$  \\
\specialrule{0.05em}{3pt}{3pt}
  ${Re}_{\mu_f}$  & $12.4927\%\pm 1.9310\%$    &$7.9987\%\pm 0.5041\%$  & $8.6146\% \pm 0.1949\%$   \\
\specialrule{0.05em}{3pt}{3pt}
 Error      &$\lambda=10$  &$\lambda=1$  &$\lambda=0.1$ \\
\specialrule{0.05em}{3pt}{3pt}
  ${Re}_{\mu_f}$   & $9.7343\% \pm 0.5150\%$  &$16.1470\%\pm 1.5965\%$
 &$19.4545\% \pm 1.2732$\%   \\
\specialrule{0.05em}{3pt}{3pt}
\end{tabular}
\end{table}

\section{Conclusion and remark}

In this work, the inverse dynamic source problem by the one single boundary measurement in finite time domain is considered, which is arising from the time-domain fluorescence diffuse optical tomography (FDOT). We establish the uniqueness theorem and  the conditional stability of Lipschitz type for the inverse problem by a defined weighted norm. A deep neural network-based reconstruction scheme with a new loss function has been proposed to solve the inverse problem. Generalization error estimates based on the Lipschitz conditional stability of inverse problems have been established, which provide a measure of the stability and accuracy of the proposed method in solving the inverse problem. Some numerical experiments have been conducted, which indicate the effectiveness of the approach in this work.

\section*{Acknowledgments}
Zhidong Zhang is supported by the National Key Research and Development Plan of China (Grant No. 2023YFB3002400), and 
National Natural Science Foundation of China (Grant No. 12101627). Chunlong Sun thanks National Natural Science Foundation of China (Grant Nos.12201298, 12274224), Natural Science Foundation of Jiangsu Province, China (Grant No. BK20210269) and  "Double Innovation" Doctor of Jiangsu Province, China (Grant No.JSSCBS20220227). Mengmeng Zhang is supported by National Natural Science Foundation of China  (Grant No.12301537), Natural Science Foundation of Hebei Province (Grant No.A2023202024) and Foundation of Tianjin Education Commission Research Program (Grant No.2022KJ102).

\bibliographystyle{plainurl} % apa abbrv
\bibliography{ref}

\end{document}